\documentclass[11pt]{amsart}

\DeclareFontFamily{U}{matha}{\hyphenchar\font45}
\DeclareFontShape{U}{matha}{m}{n}{
  <5> <6> <7> <8> <9> <10> gen * matha
  <10.95> matha10 <12> <14.4> <17.28> <20.74> <24.88> matha12
  }{}
\DeclareSymbolFont{matha}{U}{matha}{m}{n}
\DeclareFontFamily{U}{mathx}{\hyphenchar\font45}
\DeclareFontShape{U}{mathx}{m}{n}{
  <5> <6> <7> <8> <9> <10>
  <10.95> <12> <14.4> <17.28> <20.74> <24.88>
  mathx10
  }{}
\DeclareSymbolFont{mathx}{U}{mathx}{m}{n}

\DeclareMathSymbol{\obot}         {2}{matha}{"6B}
\DeclareMathSymbol{\bigobot}       {1}{mathx}{"CB}

\usepackage[pdfauthor={Congling Qiu}, 
    pdftitle={???}%
  dvips,colorlinks=true]{hyperref}
\hypersetup{
bookmarksnumbered=true,
citecolor=black,
pagecolor=black, 
urlcolor=black,  
} 

\usepackage[utf8]{inputenc}

\usepackage{xtab}
\usepackage{amsmath, amssymb, cancel,verbatim}
\usepackage[all]{xy}
\usepackage[scr=euler,scrscaled=.96]{mathalfa}

\usepackage{enumerate}
\usepackage{mathtools,booktabs}
\usepackage{amscd}
\usepackage{bbm, stmaryrd}
\usepackage{yfonts}
\usepackage{color}

\usepackage{epstopdf} 
\usepackage{booktabs}

\theoremstyle{plain}
\newtheorem{proposition}{Proposition}[subsection]
\newtheorem{conj}[proposition]{Conjecture}
\newtheorem{cor}[proposition]{Corollary}
\newtheorem{lem}[proposition]{Lemma}
\newtheorem{thm}[proposition]{Theorem}
\newtheorem{prop}[proposition]{Proposition}

\theoremstyle{definition}
\newtheorem{defn}[proposition]{Definition}

\newtheorem{asmp}[proposition]{Assumption}

\theoremstyle{remark}
\newtheorem{rmk}[proposition]{Remark}

\numberwithin{equation}{section}

\newcommand{\BA}{{\mathbb {A}}} 
\newcommand{\BC}{{\mathbb {C}}} 
 \newcommand{\BF}{{\mathbb {F}}}
\newcommand{\BG}{{\mathbb {G}}}

\newcommand{\BQ}{{\mathbb {Q}}}

 \newcommand{\BX}{{\mathbb {X}}}
 \newcommand{\BZ}{{\mathbb {Z}}}

\newcommand{\cI}{{\mathcal {I}}}

\newcommand{\cO}{{\mathcal {O}}}

 \newcommand{\cX}{{\mathcal {X}}}

\newcommand{\fm}{{\mathfrak{m}}} 
 \newcommand{\fp}{{\mathfrak{p}}}

  \newcommand{\fB}{{\mathfrak{B}}}
\newcommand{\fC}{{\mathfrak{C}}} 
 
\newcommand{\fG}{{\mathfrak{G}}} 
\newcommand{\fI}{{\mathfrak{I}}} 
 
\newcommand{\fM}{{\mathfrak{M}}}

\newcommand{\fS}{{\mathfrak{S}}} \newcommand{\fT}{{\mathfrak{T}}}
\newcommand{\fU}{{\mathfrak{U}}} \newcommand{\fV}{{\mathfrak{V}}}
\newcommand{\fW}{{\mathfrak{W}}} \newcommand{\fX}{{\mathfrak{X}}}
\newcommand{\fY}{{\mathfrak{Y}}} \newcommand{\fZ}{{\mathfrak{Z}}}

\newcommand{\wt}{\widetilde}\newcommand{\ol}{\overline}
\newcommand{\wh}{\widehat}

\newcommand{\incl}{\hookrightarrow}

\newcommand{\bsl}{\backslash}

\newcommand{\vep}{\varepsilon} \newcommand{\ep}{\epsilon}
\newcommand{\vpl}{\varprojlim}
\newcommand{\fpl}{{\flat+}}         
\newcommand{\vil}{\varinjlim}  
\newcommand{\lb}{\left(} \newcommand{\rb}{\right)}

\newcommand{\etale}{\'{e}tale~}
\newcommand{\et}{{\acute{e}t}}

\newcommand{\Art}{{\mathrm{Art}}}

\newcommand{\cycl}{{\mathrm{cycl}}}

\newcommand{\can}{{\mathrm{can}}}

 \renewcommand{\div}{{\mathrm{div}}}

\newcommand{\Frac}{{\mathrm{Frac}}}\newcommand{\Fr}{{\mathrm{Fr}}}

\newcommand{\Gal}{{\mathrm{Gal}}} \newcommand{\GL}{{\mathrm{GL}}}

\newcommand{\GSp}{{\mathrm{GSp}}}

\newcommand{\op}{{\mathrm{op}}}

\newcommand{\pr}{{\mathrm{pr}}}

\DeclareMathOperator{\Spec}{Spec}\DeclareMathOperator{\Spf}{Spf}\DeclareMathOperator{\Spa}{Spa}\DeclareMathOperator{\MOD}{mod}

\newcommand{\univ}{{\mathrm{univ}}}

 \newcommand{\Nilp}{{\mathrm{Nilp}}}
\newcommand{\zar}{{\mathrm{zar}}}

\newcommand{\CM}{CM}

\newcommand{\Nil}{{\mathrm{Nil}}}

\newcommand{\CPo}{{\mathbb{C}_p^\circ}}

\newcommand{\CP}{{\mathbb{C}_p}}
\newcommand{\Cflat}{{\mathbb{C}_p^\flat}}
 \newcommand{\Cfcc}{{\mathbb{C}_p^{\flat\circ}}}
\newcommand{\Fo}{{F^\circ}}
\newcommand{\oFo}{{\ol F^\circ}}

\newcommand\supervisor[1]{\def\@supervisor{#1}}

\newcounter{elno}

\renewcommand{\cong}{\simeq}

\author{Congling Qiu} 
\begin{document} 

\title{Linearity   on ordinary   Siegel moduli schemes and   joint   unlikely almost intersections}

\maketitle 
\tableofcontents 

\begin{abstract}
The goal of this paper is to study a $p$-adic analog of   the joint of the conjectures of Andr\'e--Oort and Andr\'e--Pink.
More precisely, on a product of ordinary Siegel formal  moduli schemes, we 
study the distribution of   points whose components are either CM points or points in   Hecke orbits. 
 We use linearity of  formal subschemes  of the product
as the $p$-adic analog of geodesicness    over complex numbers.   
Moreover, we relax the  usual incidence relations 
by using $p$-adic distance. 
  We also study a $p$-adic formal scheme theoretic analog  of the Ax--Lindemann theorem.
 \end{abstract}

\section{Introduction}


\subsection{Joint unlikely intersections}

In the theory of unlikely intersections,  a prototype conjecture is due to
Manin and Mumford, which states that an irreducible  closed subvariety of an abelian variety containing infinitely many
torsion points is a translation of an abelian subvariety.   It was proved by Raynaud \cite{Ray83} using a $p$-adic method.
The  Mordell--Lang conjecture \cite{Lang1}
is a vast generalization of the Manin--Mumford conjecture,  which replaces torsion points by the division points of a lattice. 
The Mordellic part of this conjecture  \cite{Lang},  where one takes the same statements but only with lattice points, was proved by Faltings  in his fundamental works  \cite{Fal,Fal3}  on
 Diophantine approximation, i.e., comparison of local and global heights. 
The whole conjecture was later settled by McQuillan \cite{McQ}. 
Poonen \cite{Poo} and  Zhang \cite{Zha1} generalized the  Mordell--Lang conjecture,
replacing division points by  points in their small neighborhoods, with the distance between any two points defined as the N\'eron-Tate height of their difference under the group structure\footnote{This  generalization was motivated by 
the Bogomolov conjecture (a Theorem of Ullmo \cite{Ull} and Zhang \cite{Zha}), which is 
the  generalization of  the  Manin--Mumford conjecture in this   style.}.

Now we turn to Shimura varieties and we have  two conjectures on unlikely intersections.
The Andr\'e--Oort conjecture  \cite{Andr,Oor}  is the analog of the Manin--Mumford conjecture  for Shimura varieties. It  
asserts that an irreducible  closed subvariety of a Shimura variety with a  Zariski dense     subset of CM points is 
an
irreducible component of a  
Hecke translation of   a Shimura subvariety (in particular,  totally geodesic). 
It is now a theorem   of   Pila,  Shankar and  Tsimerman \cite{PST}.
The Andr\'e--Pink conjecture   \cite{Andr,Pin}    is partially modeled on the Mordell--Lang conjecture. It asserts that an irreducible closed subvariety of a Shimura variety that has   Zariski dense     intersection with a Hecke orbit is totally geodesic.  
The conjecture  has been proved by  Richard and  Yafaev \cite{RY}   for certain Shimura varieties, including the ones of abelian type.
In these proofs, a fundamental  result, the Ax--Lindemann theorem, plays an essential role.

Since 
the Andr\'e--Pink conjecture and  the Andr\'e--Oort conjecture do not imply each other, we propose the following joint   conjecture, as a plain generalization of both conjectures.
\footnote{Conjecture  \ref{jointc} does not formally follow from the Andre--Oort conjecture and Andre--Pink conjecture. However, as pointed out to us by  Gao,   combining the aforementioned proofs of both conjectures via the   method of Pila and Zannier \cite{PZ}, one may obtain a proof of Conjecture \ref{jointc}, for Shimura varieties  of abelian type.}.  

\begin{conj} \label{jointc}Let $ S_1,S_2$ be two Shimura varieties over $\BC$ and $V\subset S_1\times S_2$ an irreducible  closed subvariety.  
Let  $O\subset S_1(\BC)$ be a    Hecke orbit,   and $\CM\subset S_2(\BC)$  the set of  CM points. 
Assume that $V(\BC)\cap (O\times \CM)$ is  Zariski dense in $V$. Then $V$ is totally geodesic.   
\end{conj}

We also want  an analog of the  
result of Poonen   and  Zhang  for Shimura varieties. 
However, due to the lack of a global group structure on Shimura varieties, such an analog does not exist\footnote{The Bogomolov conjecture  (equivalently) generalizes the Manin--Mumford conjecture replacing torsion points by  points of small heights.
However, 
CM points on Shimura varieties not necessary have small heights.  Thus an analog of the Bogomolov conjecture  in this form does not exist either.}.
Instead, we  consider the $p$-adic analog  of Conjecture \ref{jointc} for ordinary Siegel formal  moduli schemes 
where the analog of  geodesicness  is   linearity, as suggested by  Moonen \cite{Moo,Moo2}.
 We in fact conjecture a generalization of this  $p$-adic analog using $p$-adic distance (generalizing $p$-adic local height).
We prove cases of this conjecture.
 We also conjecture a $p$-adic formal scheme theoretic analog  of the Ax--Lindemann theorem (see \cite[Theorem 1.1-1.4]{UY}), and prove cases of this conjecture.



Below,  we describe the conjectures and theorems of this paper in more detail.

\subsection{Linearity}

Let $k$ be an algebraic closure of $\BF_p$, and $F$  a complete discrete valuation field of characteristic 0 with residue field $k$, equipped with the natural $p$-adic valuation. 
The valuation extends naturally to the   algebraic closure $\ol F$.   Let $\Fo\subset F$ and $\oFo\subset \ol F$ be the  valuation rings. By a formal torus over $ \Fo$, we mean  a self-product of the completion of 
the multiplicative group over $ \Fo$ along the unit.
By a    translated   formal subtorus of a formal torus over $ \Fo$, we always mean  the translation of  a \textit{nontrivial} formal subtorus 
over $ \Fo$ by an $\oFo$-point, regarding  this $\oFo$-point  as a closed   formal $\Fo$-subscheme of the torus.

Let $\fS $ be the  Siegel formal  moduli  scheme over $ \BZ_p$  of abelian schemes of relative dimension $g$ with ordinary reduction of a certain level outside $p$.   Abusing notation, we simply denote $\fS_{\Fo} $  by $\fS$ in this introduction.
By a formal subscheme of  $ \fS$, we always mean
a  locally closed formal subscheme of  $ \fS$. 
(A  formal subscheme may not extend to a  closed formal subscheme.) 
For a formal subscheme $\fV$ of  $ \fS$ and $x\in \fV(k)$, let $\fV_x$ be the formal completion of  $\fV$ at $x$. 
Since  $ \fS_x$  is  naturally a formal torus over $ \Fo$ by the Serre--Tate theory (see \ref{Local toric structure}), we shall call it a formal residue torus to emphasize the toric structure. 
We call  $\fV$  weakly linear  at  $x$  if  $\fV_x$ 
is a   union of    translated   formal subtori of $ \fS_x$. 
We call $\fV$  weakly linear if so it  is at  every $k$-point. 
See \ref{Speciality} for more discussion about this notion.

Regarding the formal residue tori as uniformizing spaces of $\fS$, 
we propose a   $p$-adic formal scheme theoretic analog  of the Ax--Lindemann theorem (see \cite[Theorem 1.1-1.4]{UY}).

\begin{conj}\label{ALconj}
Assume that $\fV$ is a closed formal  subscheme of an open formal subscheme $\fU$ of $\fS$, $x\in \fV(k)$ and $\fT$ a  translated   formal subtorus of $ \fS_x$. 
If $\fV$ is the  schematic image of $\fT\to \fU$, i.e.  the minimal   closed formal subscheme   of $\fU$ through which this morphism factors, then $\fV$ is weakly linear. 
\end{conj} 

The first main results of this paper are
\begin{thm}  \label{wwcor1} 
Conjecture \ref{ALconj} holds if $\fV_k$ is  unibranch without embedded points and $\fT$  contains a torsion point. 
\end{thm}



\begin{thm}  \label{wwcor} 
Let $\fV$ be a connected formal subscheme   of  $ \fS$ flat over $\Fo$ such that
$\fV_k$ is  unibranch.
Let $x\in \fV(k)$.

(1) If   $\fV_x$
contains a   translated   formal subtorus of $\fS_x$  of the same dimension with $\fV_x$,
 then    for  every $z\in \fV(k)$,  $\fV_{z}$   contains a   translated   formal subtorus of $\fS_{z}$ of the same dimension with $\fV_{z}$. 
 
(2)  Assume that $\fV_k$  has no embedded points.
If  $\fV$ is    weakly linear   at  $x$, then    it is weakly linear. 


\end{thm}
The analog  of  Theorem \ref{wwcor}  over $k$ is a result of  Chai \cite[(5.3)]{Chai}. We give a new proof (Proposition \ref{-11}).  
Chai \cite[(5.3.1)]{Chai} also  conjectured that the unibranch  assumption in  Proposition \ref{-11} could  be removed. Assuming this conjecture,  we can remove the  unibranch  assumption in both  Theorem \ref{wwcor1} and Theorem \ref{wwcor}.


\subsection{Joint unlikely almost intersections}   

In view of the discussion below Conjecture \ref{jointc}, we propose the following analog of  the  
result of Poonen \cite{Poo} and  Zhang \cite{Zha1}. 
\begin{conj} \label{jointp} 
Let $\CM\subset \fS(\oFo)$   be the set of CM points in $ \fS(\oFo)$.
Let  $O$ be  a  Hecke orbit \footnote{As pointed out by  Liu, it could be complicated to define   nonprime-to-$p$ Hecke action  using   moduli  schemes with levels at $p$. A convenient definition is by using the usual Hecke action over $\overline F$  which preserves ordinariness. However, our Hecke--Frobenius orbits below are always defined using   moduli schemes.} in $ \fS(\oFo)$.    
Let $\fV$ be  a  formal subscheme of $\fS^{2}$.    
Let $Z\subset \fV_k$ be a closed subscheme contained in the Zariski closure (with reduced induced structure) of 
the reduction of $\fV_{\ep}  \cap  (O\times \CM)$ for all  $\ep>0$.  Then $\fV$ contains a union of weakly linear  closed  formal subschemes that  contains $Z$.

\end{conj} 
Here $\fV_{\ep} \subset \fS^2(\oFo)$ is the subset of points of $p$-adic distance $\leq \ep$ to $\fV$, see
\ref{distance}. 
Note that $Z$ could be larger than the Zariski closure  of     the reduction of $\fV  (\oFo) \cap  (O\times \CM)$.

Let us justify our formulation of Conjecture \ref{jointp}.
Define  
$(O\times \CM)_\ep\subset \fS^2(\oFo)$ to be the subset of points of $p$-adic distance $\leq \ep$ to some point in $O\times \CM$.
Then  $\fV(\oFo)  \cap  (O\times \CM)_\ep\subset \fV_{\ep}  \cap  (O\times \CM)$.
Conjecture \ref{jointp}
 would  be a more strict analog   of Poonen \cite{Poo} and  Zhang \cite{Zha1}  if we   replace  ``the reduction of $\fV_{\ep}  \cap  (O\times \CM)$"
by ``$\fV(\oFo)  \cap  (O\times \CM)_\ep$", and consider its schematic closure. However,   even
the schematic   density  of $\fV (\oFo) \cap  (O\times \CM)$ in $\fV$  
may not imply the weak  linearity of $\fV$.   Indeed, 
the naive generalization of the Mordell--Lang conjecture for formal tori is wrong, see \cite[2.3]{Ser22}.

Assuming   Conjecture \ref{ALconj}, to prove  Conjecture \ref{jointp}, it is enough to show that for some   $x\in Z(k)$, $\fV_x$ contains a union of  translated formal subtori  of   $\fS_x^{2}$ that  contains   $Z_x$ as a closed formal subscheme.
Another main result of this paper (Theorem \ref{jointthm}) is
 toward this weaker form of Conjecture \ref{jointp}. To state it, we need some more notions.  
By the ``Frobenius" on $\fS$, we mean the ``canonical lifting" of the (relative) Frobenius endomorphism on $\fS_k$, see \ref{Lifts of  Frobenii}. The Frobenius action is a special $p$-primary Hecke action.
By a   Hecke--Frobenius orbit on $\fS$, we mean an orbit under the prime-to-$p$ Hecke action and  (forward and backward)   Frobenius action, see \ref{Hecke action}.  
A   Hecke--Frobenius orbit is  contained   in a unique Hecke orbit. 
If $g=1$, then a   Hecke--Frobenius orbit is the same as a Hecke orbit.

\begin{thm} \label{jointthm} 
Let $d,d'$ \footnote{Letting  $d'=1$, its is not hard to see  that we   get an equivalent  statement using embeddings of Siegel moduli schemes. The current  statement   is more convenient  for our purpose. However, the case $d=1$   is weaker than the general case in which  we have   Frobenii on all components, not just the simultaneous Frobenius.} be  non-negative integers and $\fV$    a  closed formal subscheme of $\fS^{d+d'}$.  For $i=1,...,d$, let  $O_i$ be  a  Hecke--Frobenius orbit  in $ \fS(\oFo)$.

(1)  Let $e$ be a positive integer.
Let $Z\subset \fV_k$   be a  closed subscheme contained in the Zariski closure of 
the reduction of $$\fV_\ep \cap  \lb \prod_{i=1}^dO_i\times \CM^{d'}\rb \cap \bigcup_{[E:F]\leq e}\fS (E^\circ)$$ for all  $\ep>0$.  Then for every $x\in Z(k)$, $\fV_x$ contains a union of  translated formal subtori  of   $\fS_x^{n+d'}$ that  contains   $Z_x$ as a closed formal subscheme.
     
(2)    Assume that  the   reduction of $\fV_{\ep}  \cap  \lb \prod_{i=1}^dO_i\times \CM^{d'}\rb$ is Zariski dense in  (the underlying topological space of) an irreducible component  $\fV_k$
for all  $\ep>0$. There is  a nonempty open subscheme of $\fV_k$ such that for every   $x$ of its $k$-points, $\fV_{x}$   contains a translated formal subtorus  of   $\fS_x^{d+d'}$.

\end{thm}

 

Besides the  result of Poonen \cite{Poo} and  Zhang \cite{Zha1}, another motivation for  Theorem \ref{jointthm} is our previous work   \cite{Qiu},  where we proved that  CM points outside a   formal subscheme can not be $p$-adically too close to it.  
 Theorem \ref{jointthm} is an attempt   to include  Hecke orbits.

 \subsection{Structure and strategy}\label{Structure and strategy}

In Section \ref{Igusa schemes}, we introduce  two Igusa (formal) schemes   with   natural    projections $\fI\to \fI_0\to \fS$.  There is a   toric action on $\fI_0$.  

In Section \ref{Linearity}, we use the  toric action on $\fI_0$   to  study  linearity.  We prove Theorem \ref{wwcor}.  
We also prepare some technical lemmas    for    later use.

In Section \ref {Frobenii and linearity},   we introduce the  Frobenii on these formal schemes.
Then, we further discuss  linearity. 
We end this section by proving Theorem \ref{wwcor1} using  a result of de Jong \cite{dJ} 
that relates Frobenius-invariance  and  linearity for formal tori.

In  Section \ref{More notions},  we  define more notations used in Theorem  \ref{jointthm} 
and its proof. In particular, we introduce  the  adic generic fiber   of $\fI$, which is a perfectoid  space. 
 
Theorem  \ref{jointthm} (1)  is proved at the end of Section \ref{Bounded ramification}.
The main strategy for the proof is to use the toric action on $\fI_0$ to reduce Theorem  \ref{jointthm} (1) to an analog  for  $\fI_0$ and canonical liftings, which says that for  a set   of canonical liftings approaching  a   formal subscheme   $\fW\subset \fI_0$, 
density of the set of their reductions  implies  local Frobenius-invariance of $\fW$.
(Here and below, canonical liftings are units of  the formal residue tori of $\fI_0$ (and $\fS$), and  are CM points.)
In \ref{unt}, we apply    some perfectoid technique  in our previous work \cite{Qiu} (originated from the work of Xie \cite{Xie}) to the  adic generic fiber   of $\fI$   to prove this analog on $\fI_0$. It in fact holds for a power of $\fI_0$.
In \ref{Unramified points on Siegel moduli schemes}, using this result for $\fI_0^{d+d'}$ and the  toric action on the first $d$ factors,
which commutes with the prime-to-$p$ Hecke action,
we prove a weaker version  of Theorem  \ref{jointthm} (1), where each $O_i$ is a prime-to-$p$ Hecke orbit and $\CM$ is replaced by the set of canonical liftings.  
In \ref  {Forward Frobenius action},  using the fact that forward Frobenius action   shrinks the formal residue tori toward canonical liftings and backward Frobenius action increases ramification, 
we reduce Theorem  \ref{jointthm} (1) to the weaker version.

In Section \ref{Frobenius action},  we prove Theorem \ref{jointthm} (2) by ``globalizing"  a trick of Boxall \cite{Box}, used by Serban  
\cite{Ser21}  \cite{Ser22}  for formal groups,  to include all ramified points.  The ``globalization" is achieved using the Igusa scheme $\fI_0$ and   some lemmas about $\fI_0$ prepared in  Section \ref{Linearity}.

   \subsection*{Acknowledgements}            The author  thanks   Ziyang Gao and Yifeng Liu   for   their helpful comments on the paper. He thanks   Brian Conrad  and Vlad Serban for the inspiring  discussions.      The author  was motivated by Daniel Kriz's talks  in the spring of  2021 to study Igusa schemes.        The research   is partially supported by the NSF grant DMS-2000533.

\section{Igusa schemes}\label{Igusa schemes}

\subsection{Siegel moduli schemes}\label{Ordinary perfectoid Siegel space}

Let $\BG_m$ be the multiplicative group over $\BZ_p$ and  $\wh \BG_m$  the formal completion of $\BG_m$ along the unit. We identify   $\wh \BG_m$ with its   $p$-divisible group via  the Serre--Tate theorem. For an abelian scheme $A$, let  $\wh A$ be the formal completion  of $A$ along the unit. 

Now we define an ordinary Siegel formal  moduli scheme.
Let $\Nilp^\op_{\BZ_p} $  be  the opposite  category of $\BZ_p$-algebras with $p$ nilpotent. For  $R\in \Nilp^\op_{\BZ_p} $, an abelian $R$-scheme $A$  is 
ordinary if   $\wh A[p^\infty]\cong \wh\BG_m^g$ and $A[p^\infty]/\wh A[p^\infty] \cong (\BQ_p/\BZ_p)^g$. 
Let $N\geq 3$ be an  integer  prime to $p$. For a  principally polarized $g$-dimensional abelian
scheme   $A$ over a $\BZ_p$-scheme,  a  level-$N$-structure  is    a symplectic  isomorphism  of finite group schemes 
$
(\BZ/N\BZ)^{2g}\cong  A[N]
$
where  the skew-symmetric form on $(\BZ/N\BZ)^{2g}$  is  the standard  one, and on  $ A[N]$ is the one induced from  the Weil pairing.  
The functor    assigning to  $R\in \Nilp^\op_{\BZ_p} $     the set of isomorphism classes of  principally polarized  ordinary  abelian $R$-schemes 
with level-$N$-structure is representable by a formal scheme $\fS=\fS_{N} $  topologically of finite type  and flat   over $ \BZ_p$.  
(One may use   quotient   to obtain more general level structures.)

We relate $\fS$ to the usual Siegel moduli schemes. Let $ S= S_{N}  $  over $\Spec \BZ_{(p)}$ be 
the 
moduli space of principally polarized $g$-dimensional abelian
schemes over $\BZ_{(p)}$-schemes with level-$N$-structure.      
Remove the nonordinary locus in the special fiber of $S$. Then $\fS $  is  the  $p$-adic completion of the resulted scheme.

\subsection{Igusa schemes}\label{Igusa schemes and Hecke orbits}   
Consider the $p$-divisible group $\BX=\wh\BG_m^g\times (\BQ_p/\BZ_p)^g$ over $\BZ_p$, with its connected part  $\wh \BX=\wh\BG_m^g$ and \etale part $ \BX_\et=(\BQ_p/\BZ_p)^g$.  

Fix the canonical polarization $\BX\cong\BX^\vee$, $\wh \BX \cong\BX_\et^\vee$ and $  \BX_\et \cong \wh\BX^\vee$.  For an integer $i\geq 0$ or $i=\infty$, let $\fS_i/\Spf \BZ_p$ be the functor    assigning to  $R\in \Nilp^\op_{\BZ_p} $     the set of isomorphism classes of triples 
\begin{equation}\fS_i=\label{inter}\{(A,\vep_0,\vep_1):A\in \fS(R),\ \vep_0:\wh A[p^i] \cong  \wh\BX_R[p^i], \ \vep_1:A[p^i]/\wh A[p^i] \cong \BX_{\et,R}[p^i]\}/\sim,\end{equation}
where  the isomorphisms $\vep_0$ and $\vep_1$ are required to preserve the filtered polarizations
$\wh A[p^i]\to (A[p^i]/\wh A[p^i])^\vee$ and  $ A[p^i]/\wh A[p^i] \to \wh A[p^i]^\vee$
up to a common scalar. 
By \cite{Hid}, the functor $\fS_i$ is represented by a formal scheme 
(which we still denote by $\fS_i$ so that $\fS_0=\fS$).
For $j\geq i$, we have 
the natural projection
$f_{j,i}: \fS_j\to  \fS_i$, which is surjective finite \'etale if $j<\infty$. And 
$$\fI_0:=\fS_\infty$$ is the inverse limit of $\fS_i$'s (in the category of   formal $\BZ_p$-schemes). 
Below, we only use $f_{j,i}$ for $j<\infty$, and    $$\pi_i:=f_{\infty,i}:\fI_0\to \fS_i.$$

Let $\fI/\Spf \BZ_p$ be the functor  assigning to $R\in \Nilp^\op_{\BZ_p} $   the set of isomorphism classes of pairs 
$$\fI=\{(A,\vep):A\in \fS(R),\ \vep:A[p^\infty]\cong \BX_R\}/\sim,$$
where the isomorphism $\vep$ is required to preserve the polarization up to a scalar. Then we have natural projections 
\begin{equation}\fI \xrightarrow{\pi} \fI_0 \xrightarrow{\pi_0} \fS .\label{pi0}\end{equation}

\begin{prop}[{\cite[Proposition 4.3.3, Corollary 4.3.5, Lemma 4.3.10]{CS}}]\label{CS3}
(1) The functor $\fI$ is representable by a formal scheme in the sense of \cite[2.2]{SW} (which we still denote by $\fI$). 

(2) The special fiber  $\fI_{\BF_p}$ is   the perfection of $ \fI_{0,\BF_p}$.

(3) The formal  scheme $\fI$ is the Witt vector scheme of $\fI_{\BF_p}$, i.e., it is the unique lift of $\fI_{\BF_p}$ to a flat $p$-adic formal scheme  over $\BZ_p$.

\end{prop}

\begin{lem} \label{milne}  
For $y\in \fI_{0}(k)$ and a positive integer $n$,   the local ring $ \cO_{\fI_{0,/p^n},y}$ is noetherian.
\end{lem}
\begin{proof} Note that $\fI_{0,/p^n}$ is a scheme and is the inverse limit of $\fS_i/p^n$'s. 
The proof is the same as  the proof of  the noetherianness in \cite[Proposition 2.4]{Mil}.
\end{proof}

   \subsection{Local toric structure}\label{Local toric structure}
Let $k$  be an algebraic  extension of $\BF_p$.
Let $W$ be the ring of Witt vectors  of $k$.   Let $\fM/\Spf W$ be  the  height-0 component of the Rapoport-Zink space of $ \BX_{k}$, i.e., for $R\in \Nilp^\op_W$, $\fM(R)$ is  the set of isomorphism classes of pairs $(X,\vep')$ where $X$ is a principally  polarized $p$-divisible group over $R$, and $\vep:X_{R/p}\to  \BX_{R/p}$ is a polarization-preserving (up to a scalar) quasi-isogeny of height 0.  
Then  $\fM$ is representable by a formal scheme locally formally of finite type  over $ W$ (see \cite[Theorem 6.1.2]{SW}). In particular, $\fM$ is determined by its restriction  to the subcategory $\Art^\op_W$ of $\Nil^\op_W$ of artinian local $W$-algebras with residue field ${k}$.
Then by the Serre--Tate  theory \cite{Kat} \cite{CO} , 
$\fM/\Spf W$ coincides with the polarized  deformation space of $ \BX_{k}$ over $\Art^\op_W$.
By \cite[2.18, 2.19]{CO} or \cite[p 150]{Kat}, $\fM(R)$ equivalently classifies equivalence classes of polarized extensions
\begin{equation}0\to   \wh\BX_R\to G \to  \BX_{\et,R}\to 0.\label{RAR1}\end{equation}
In particular  we have the universal object
\begin{equation}\label{3.3'}0\to   \wh\BX_\fM\to G^\univ \to  \BX_{\et,\fM}\to 0. \end{equation}
There is  a formal torus structure on  $\fM$ defined by Baer sum (see \cite[Theorem 2.19, Corollary 2.24]{CO}):     \begin{equation*}\fM\cong \wh\BG_{m,W}^{g(g+1)/2}.\end{equation*}
The  unit of $\fM$ corresponds to $\BX_W$ with the canonical identification $(\BX_{W })_{W/p}=\BX_{W/p}$.

For $   {  x}\in \fS({k})$, $\pi_0^{-1}(\{x\})\subset \fI_{0,k}$  is a union of $k$-points. If $x$
corresponds  to a polarized ordinary  abelian variety $A_{    x}$ over ${k}$, $y\in \pi_0^{-1}(\{x\})$ 
corresponds  to the choice of a pair polarization-preserving  isomorphisms
$\wh A_{x}\cong \wh \BX_{k}$ 
and      $A_x/\wh A_{x} \cong \wh \BX_{\et,{k}}$. 
Then we have an isomorphism 
$$a_y:\fS_{{    x}}\cong \fM$$
that sends $A\in\fS_{{    x}}(R) $, $R\in \Art_W^\op$, to a polarized extension \begin{equation}0\to   \wh\BX_R\to A[p^\infty] \to  \BX_{\et,R}\to 0 \label{RARa}\end{equation}
determined  by the above two isomorphisms over ${k}$ (see \cite[2.18, 2.19]{CO} or \cite[p 150]{Kat}). 
For a different $    y$, $a_y$ differs by an automorphism of $\fM$. (The automorphism  could be made explicit. However, we will not discuss it here.)  In particular, we have a natural formal torus structure on $\fS_x$ via  $a_{    y}$, independent of the choice of $y$.

For ${    y}\in \fI_0({k})$, let $\fI_{0,{    y}}$ be the formal completion of $\fI_{0,W}$ at ${    y}$.  
Then we have   natural isomorphisms of formal schemes
\begin{equation}b_{    y}:\fI_{0,{    y}}\cong \fS_{\pi_0(    y)} ,\ c_{    y}:\fI_{0,{    y}}\cong \fM \label{bcy}\end{equation}
defined as follows.
By Lemma \ref{milne}, 
we only need to define the morphisms on  $R$-points for $R\in \Art^\op_W$.
By the discussion in the last paragraph, a point  $  \wt y\in \fI_0(R)$ with reduction $    y$  corresponds to a pair: $A\in \fS_{\pi_0(    y)}(R)$ and 
an extension \eqref{RARa}.
Let  $b_{    y}(\wt y)=A$   (i.e., $b_{    y}=\pi_0|_{\fI_{0,{    y}}}$), and  let   $c_{    y}(\wt y)$ be \eqref{RARa}.
Then  by    the discussion  in the last paragraph, $b_{    y}$ and  $c_{    y}$ are  isomorphisms. It is also easy to check  that $c_{    y}\circ  b_{    y}^{-1}  =a_y$. 
In particular, we have a natural formal torus structure on  via  $c_{    y}$, and $b_y$ is an isomorphism of formal tori.  We fix these  local formal torus structures on $\fI_0$ and $\fS$
from now on.


 \subsection{Toric action}\label{Toric action}
We define an action of $\fM$ on $\fI_{0,W}$:
$$\fM\times\fI_{0,W}\to \fI_{0,W},\ (h,x)\mapsto hx$$
following \cite[Proposition 2.3.5]{LZZ}. The proof in  \cite[Proposition 2.3.5]{LZZ} is not rigorous (see also \cite[1.2.10]{How}), and  we modify it as follows. 
For $R\in \Nilp^\op_W$ (not necessarily  in  $\Art^\op_W$) and  $h\in \fM(R)$, 
 we  
have a polarized extension \eqref{RAR1} for $G$ coming from base change of  the universal object \eqref{3.3'}. 
For 
$x\in \fI_0(R)$,    let  $x$  correspond to  $A\in \fS(R)$ and 
a polarized extension
\eqref{RARa}.  The Baer sum  of \eqref{RAR1} and \eqref{RARa} is 
a polarized extension
\begin{equation}0\to   \wh\BX_R\to G' \to  \BX_{\et,R}\to 0.\label{RAR}\end{equation}
From the polarization-preserving quasi-isogeny  $ G_{R/p}\to  \BX_{R/p}$  of height 0, 
we have a  polarization-preserving quasi-isogeny $G'_{R/p}\to A[p^\infty]_{R/p}$ of height 0.
It lifts uniquely to a polarization-preserving quasi-isogeny  $G' \to A[p^\infty] $,
by the rigidity theorem for p-divisible groups up to isogeny \cite[LEMMA 1.1.3]{Kat}.
Then the Serre--Tate  theory  \cite[THEOREM 1.2.1]{Kat}, there is a unique   $A'\in \fS(R)$, with a $p$-primary quasi-isogeny to  $A$ and the restriction of the $p$-primary quasi-isogeny to the $p$-divisible groups is identified with  $G' \to A[p^\infty]  $.
Combined with \eqref{RAR}, we get an element in $\fI_0(R)$, which we denote by $hx$. 

   The following lemma is easy to check by definition.

 \begin{lem} \label{Via} The action of $\fM$ stabilizes   $\fI_{0,y}$.      Via $c_y$,  the action of $\fM$ on $\fI_{0,y}$  is identified with the multiplication on $\fM$, equivalently  on $\fI_{0,y}$, i.e., $hx=c_y^{-1}(h)+x$.     \end{lem}

\begin{rmk}\label{rmk1} 

 For $g=1$, such an action was also defined by   Howe   \cite{How} for a slightly different Igusa scheme. 
 A subtlety is that    the     $\GL_2(\BA_f^p)$-action in loc. cit. differs from our $\GL_2(\BA_f^p)$-action (see \ref{Hecke action} below) by a suitable twist. This   was shown in an earlier version of   \cite{How}.

\end{rmk}



  \section{Toric action and linearity}\label{Linearity}
 In the rest of this paper, let $k$ be    an algebraic closure of $\BF_p$. 
Let $F$ be a finite extension of  $\Frac W$.   
Let $\varpi$ be a uniformizer of  $F^\circ$.

  %
  In this section, we apply the   action  of $\fM$ on  the Igusa scheme  $\fI_0$  in \ref{Toric action} to study the linearity of subschemes of $\fS_k$ and $\fS$.  The first two subsections are preparations.
Then we prove Theorem \ref{wwcor}.  
In the final subsection, we prepare some technical lemmas    for    later use.





         
                   

 \subsection{Preliminaries}\label{Preliminaries}
 We need some preliminaries on commutative algebra.

In order to deal with  formal schemes that are  not necessary noetherian, we clarify the definitions. We will only deal with  affine formal schemes in the non-noetherian case.
A closed formal subscheme of an affine formal scheme $\Spf A$   is  of the form $\Spf A/I$ where $I$ is a closed ideal so that  $A/I$ is complete  and  separated  with the quotient topology (see \cite[23.B, 23.D]{Mat}).
Note that every ideal is  closed if $A$ is noetherian, see \cite[24.B]{Mat}.
For   a morphism  $f:\Spf B\to \Spf A$,   corresponding to a ring homomorphism $f^\sharp:A\to B$, 
define the  pullback or   schematic preimage $$f^{-1}(\Spf A/I)=\Spf B/ \ol{IB},$$ where the overline denotes taking closure (in $B$), and 
define the schematic image $$f(\Spf B)=\Spf A/\ker f^\sharp.$$
Let $\Lambda$ be an index set, and for $\lambda\in \Lambda$, let $f_\lambda:\Spf B\to \Spf A$  be a morphism and regarded as a $B$-point of $\Spf A$. Define schematic closure of $\{f_\lambda\}_{\lambda\in \Lambda}$ to be 
the schematic image of $$\prod_{\lambda\in \Lambda} f_\lambda:\prod_{\lambda\in \Lambda} \Spf B\to A .$$ We say  that $\{f_\lambda\}_{\lambda\in \Lambda}$ is 
schematically dense if the schematic image is $\Spf A$.

For   a morphism  $f:\fY\to \fX$ of locally noetherian formal schemes, 
pullback of  a formal scheme over $\fX$ is defined as the fiber product. 
The schematic image $f(\fY)$ of $f$  is the minimal   closed formal subscheme   of $\fX$ though which $f$ factors. 
    (The uniqueness is clear and the existence is \cite[Lemma 2.8]{Kap}.)       We call $f$   schematically surjective if   $f(\fY)=\fX$. 
      In the affine case,  schematic image  is given by the discussion in the  last paragraph.
      Understanding affine or locally noetherian  schemes as formal schemes (with discrete topology on the rings), the above definition coincides with the usual one \cite[2.5]{BLR}\cite[Section 01R5]{Sta}.

   The formation of schematic image under a proper morphism   is compatible with flat pullback.  More precisely,   we have the following lemma.
\begin{lem}[{\cite[Proposition 2.10]{Kap}}]\label{Kaplem} let $f: \fY\to \fX$ and $g: \fX' \to  \fX$ be morphisms of locally noetherian formal schemes, where $f$ is proper and $g$ is flat. 
Let $\fY'=\fY\times_\fX\fX'$ with induced morphism $f:\fY'\to \fX'$. 
Then  $f(\fY')=g^{-1}(f(\fY))$. 
\end{lem}
For a locally noetherian formal scheme $\fX$,   the
dimension $\dim(\fX)$ of $\fX$ is defined as the supremum of the Krull dimensions of its local
rings. If $\fX$ is affine, then $\dim(\fX)$ is equal to the Krull dimension of
$\cO_\fX(\fX)$.
             \begin{lem}\label{samedim}
Let  $f:\fY\to \fX$ be        a schematically surjective finite morphism  of noetherian formal   schemes.

(1) For a closed    subscheme $\fZ\subset \fX$ (defined by an open ideal), $f|_{f^{-1}(\fZ)}:f^{-1}(\fZ)\to \fZ$ is set theoretically  surjective.

(2)  We have   $\dim \fY=\dim \fX$.

(3) For a closed point $x$ of $ \fX$, $f:\fY_{f^{-1}(x)}\to \fX_x$  is  a schematically surjective finite morphism.
(Note that $f^{-1}(x)$ is a finite union of closed points of $ \fY$, and 
$\fY_{f^{-1}(x)}$ is defined in the obvious way.)
In particular,
$\dim \fY_{f^{-1}(x)}=\dim \fX_x.$   

\end{lem}
\begin{proof} 
It is enough to consider the affine case and  assumes that $\fX,\fY$ are schemes. 
Since a finite morphism is quasi-compact and closed, by \cite  [Lemma 01R8]{Sta},  $f$ is  set theoretically surjective.  Then so is  $f|_{f^{-1}(\fZ)}$.   
(2) is    \cite[Lemma 0ECG]{Sta}.  
Since completion preserves  finiteness    \cite[Lemma 0315 (4)]{Sta} and  
injectivity (assuming noetherianness) by the Artin-Rees lemma, 
(3) follows. 
\end{proof}

   \begin{cor}\label{samedimcor}
Let  $f:\fY\to \fX$ be        a schematically surjective finite adic morphism  of noetherian formal   $\Fo$-schemes.

(1)   The induced map $\fY(\oFo)\to \fX(\oFo)$  is surjective.

(2)  Assume $\fY=\Spf B$ where $B$ is a complete local domain over $\Fo$  with maximal ideal $\fm$ 
such that the structure morphism $\Fo\to B$ is injective  and $\Fo/(\Fo\cap \fm)\to B/\fm$ is an isomorphism. Then 
$ \fY(\oFo)\neq \emptyset$. 
\end{cor}
\begin{proof}  
(1) It is enough to consider the affine case  so that  $\fX=\Spf A$
and
$\fY=\Spf B$. Any $x\in \fX(\oFo)$ is  given by $A/I\cong E^\circ$ where $E/F$ is a finite extension.
Regard 
$x=\Spec E^\circ$ as a closed subscheme of $\Spec A$. By Lemma \ref{samedim}
(1),  $\Spec B/IB\to \Spec E^\circ$ is set theoretically surjective. Since $B/IB$ is finite over $E^
\circ$,  the  underlying reduced subscheme  of an irreducible component of $\Spec B/IB$ that is surjective to 
$ \Spec E^\circ$ is of the form $ \Spec K^\circ$ where $K/E$ is a finite extension. 
Since $A\to B$ is adic, it is easy to check that $B\to K^\circ$ is continuous. This gives a $K^\circ$-point of $\fY$ whose image is $x$.

(2) follows from  (1) and  the Noether normalization lemma for complete local domain \cite[Lemma 032D]{Sta}.
Note that the Noether normalization morphism being adic is shown in the proof of  \cite[Lemma 032D]{Sta}.
\end{proof}

We prepare a lemma on (adic) completion.
\begin{lem}\label{-1lem}

                 Let $R$ be a ring and    $M,N$ two  $R$-modules. 
                 
        (1)    Let $I\subset R$ be an ideals.      A surjection   $M\to N$
between                       $R$-modules with kernel $K$ induces a surjection $\wh M_I\to \wh N_I$ between                     the $I$-adic completions, whose kernel   is the closure of the image of $K$ in $\wh M_I$.


(2) Assume that $N$ is finitely presented. Let $\{M_n\}_{n=1}^\infty$  be submodules of $M$ such that $M_{n+1}\subset M_{n}$. 
Let $\wh M=\vpl_n M/M_n$ and $\wh {M\otimes_R N}=\vpl_n {M\otimes_R N}/[M_n\otimes_R N].$
Here for a submodule $M'\subset M$,   $[M' \otimes _R N]$ is the image of ${M' \otimes _R N}$ in  ${M \otimes _R N}$. 
Then the natural morphism 
$\wh M \otimes _R N\to \wh {M\otimes_R N}$  is an isomorphism.

\end{lem}
\begin{proof}         
(1) is \cite[23.I, Poposition]{Mat} (see also  \cite[Lemma 0315 (2)]{Sta}). 
Now we prove (2).
By the right exactness of tensor,  for a submodule $M'\subset M$, $M/M'\otimes _R N\cong M \otimes _R N/ [M' \otimes _R N]$. 
Thus we only need to show that $\wh M \otimes_R N\cong \vpl_n (M/M_n \otimes_R N)$.
Let $N$ be the cokernel of a morphism  $\Phi:R^i\to R^j$. From $\Phi$, we have   induced morphisms $\phi:\wh M^i\to \wh M^j$ and $\phi_n: (M/M_n)^i\to (M/M_n)^j$.
We need to prove the isomorphism between the cokernel of $\phi$  and 
the  inverse limit of the cokernels of   $\phi_n$'s. By the surjectivity of $M/M_m\to M/M_n$ for $m\geq n$, we have  the surjectivity of the morphism  from the image of $\phi_m$ to the image of $\phi_n$. The expected isomorphism  follows from the Mittag--Leffler lemma.
          \end{proof}

We prepare an important lemma.
\begin{lem}\label{adicinj}Let $R$ be a ring,  $\alpha\in R$  primary (i.e., $\alpha R$ is a primary ideal)  and   not a zero divisor.

(1)  
For every positive integer $n$, $\alpha^n$ is primary.

(2)  Assume that $R$ is noetherian.
Let   $I\subset R$ be an ideal containing $\alpha$,  $\wh R_I$ the  corresponding $I$-adic  completion.     The natural map   $R/\alpha^n\to \wh R_I  /\alpha^n$ is injective  for every  $n$.

\end{lem}
    \begin{proof}  
    (1)   We do induction on $n$. The case $n=1$ is trivial. For a general $n>1$, assume that $\alpha^n|ab$ and $\alpha^n\nmid a$, we will show that $\alpha^n$ divides a power of $b$. Since $\alpha$ is primary, if $\alpha\nmid a$, $\alpha$ divides a power of $b$. Taking $n$-th power, we are done. 
     Otherwise, $\alpha|a$ and write $a=\alpha a_1$. Then $\alpha^n|\alpha a_1b$. Since   $\alpha$ is not a zero divisor, 
$\alpha^{n-1}|a_1b$.  Since $\alpha^{n}\nmid a=\alpha a_1$, $\alpha^{n-1}\nmid a_1$.
By the induction hypothesis, $\alpha^{n-1}$ divides a power of $b$.       So $\alpha^{n(n-1)}$ divides a power of $b$. Since  $\alpha^{n}|\alpha^{n(n-1)}$,  we are done. 
    
(2)    
    Let $x\in R$ such that the image of $ x $ in $ \wh R_I=\vpl _m R/I^m$ is in $\alpha^n \wh R_I $. Then $x\in \alpha^n  R +I^m$ for  for every positive integer $m$, i.e.,   $[x]\in[I]^m$ for  the images $[x],[I]$ of $x,I$ in $R/\alpha^n$.         Since  $\alpha^n$ is primary, by  \cite[p 144, Theorem 2]{VdW} (which requires $R$ to be noetherian), $\bigcap_m [I]^m=\{0\}.$ 
So $[x]=0$. (2) follows. 
                        \end{proof}

   \subsection{Pullback   to $\fI_0$}\label{Consequences}
         Let $ \fV=\fV_0\subset \fS_{\Fo}$ be  a   formal  subscheme. 
      In Lemma \ref{basic} below,    we want to study the pullback of $\fV$ to $\fI_0$ which is non-noetherian.  Before that, we consider an easier analog of  Lemma \ref{basic} in the noetherian case.

\begin{lem}\label{basic0} 
Assume that $\fV$ is affine,  $\fV_k$ is irreducible and has no embedded points. For    a closed formal subscheme $\fX\subset  \fV$   and  $x\in \fX(k)$ (then 
$\fX_x$ is a closed formal subscheme of  $  \fV_{x}$ by Lemma \ref{-1lem} (1)),
if
$\fX_x=  \fV_{x}$,  then $\fX=  \fV$.
\end{lem}  

\begin{proof}  
Claim: the natural map    $\cO_{ \fV}( \fV)\to  \wh\cO_{\fV,x}$   is injective.  (Here $\wh\cO_{\fV,x}$ is the  complete local ring at $x$.)
Thus the composition of the diagram of  natural morphisms $$\cO_{ \fV}( \fV)\to \cO_{\fX}(\fX)\to  \wh\cO_{\fX,x} \cong \wh\cO_{\fV,x}$$
is injective (the condition in the lemma gives the last isomorphism).   Then   the first morphism in this   diagram  is injective (also surjective by definition), and thus an isomorphism.
So   $\fX=  \fV $.

Now we prove the claim.    Since $\fV$ is flat over $F^\circ$,   $\varpi\in\cO_{\fV}(\fV)$ is not a zero-divisor. Since   $\fV_k$ is irreducible and has no embedded points,     $\varpi\in\cO_{\fV}(\fV)$ is primary. 
By Lemma \ref{adicinj} (2)  (with $R=\cO_{\fV}(\fV)$, $\alpha=\varpi$ and $I$   the   ideal of $\cO_{\fV}(\fV)$ defining $x$),  $\cO_{\fV}(\fV) /\varpi^n\to \wh\cO_{\fV,x} /\varpi^n$   is
injective.
Since $\cO_{ \fV}( \fV) $ 
and  $\wh\cO_{\fV,x}$ are $\varpi$-adically complete (see  \cite[Lemma 090T]{Sta}) and taking limits of   inverse systems is left exact, the claim follows.
\end{proof}

    Now we consider     the pullback of $\fV$ to $\fI_0$. 
Let  $f_{j,i}: \fS_j\to  \fS_i$ (for $j\geq i\geq 0$) be the surjective finite \'etale morphism as in  \ref{Igusa schemes and Hecke orbits}.
Let $ \fV_0=\fV$ 
and assume it to be  connected.
 Let $\fV_{i+1}$  be  a  connected    component  of  $f_{i+1,i}^{-1}(\fV_i)$ inductively.  
 Let   $\wt \fV$    be the inverse limit of $\fV_{i}$'s, i.e.,  $\cO_{ \fV}( \fV) $ 
is the $\varpi$-adic 
completion of $\vil_i\cO_{\fV_i}(\fV_i) $. Then $\wt \fV$  is naturally a closed   formal subscheme of  
$\pi_0^{-1}(\fV) $ where $\pi_i :\fI_0\to \fS_i$ (and $\fS_0=\fS)$  is   also defined in as in  \ref{Igusa schemes and Hecke orbits}.
Moreover  $\wt \fV_k$    is the inverse limit of $\fV_{i,k}$'s, i.e.,  $\cO_{ \fV}( \fV)/\varpi $ 
 is $\vil_i\cO_{\fV_i}(\fV_i)/\varpi $.
  Since a finite \etale morphism is open and closed, every $f_{j,i}|_{\fV_{j,k}}$ is surjective  to  $\fV_{i,k}$. 
So $\pi_0|_{\wt \fV_k}$ is surjective to $\fV_k$.
\begin{lem}
For   $x\in \fV(k)$, let  $y\in  \pi_0^{-1}  (\{x\})\cap \wt \fV$. Then
  the   isomorphism  $\pi_0|_{\fI_{0,{    y}}}:\fI_{0,{    y}}\cong \fS_{x}$    (see \eqref{bcy} and the discussion below it) induces an isomorphism 
\begin{equation}\label{indisom}
\wh\cO_{\wt\fV,y} \cong \wh\cO_{\fV,x}.
\end{equation}
\end{lem}

\begin{proof}Let $x_i=\pi_i(y)$. Since $f_{j,i}: \fS_j\to  \fS_i$ is  \'etale, the composition
 $f_{j,i}|_{ \fV_j}:\fV_j\to f_{j,i}^{-1}(\fV_i)\to  f_{j,i}^{-1}(\fV_i)$ is \'etale, where the first morphism is an open embedding. 
Since $k$ is algebraically closed, $f_{j,i}|_{ \fV_j}$   induces an isomorphism  between the 
quotients by  powers of maximal ideals
$ \cO_{ \fV_j}/\fm_{x_j}^n\cong   \cO_{\fV_i}/\fm_{x_i}^n$. The rest follows routinely  from definition.
\end{proof}

\begin{asmp}\label{basicasmp}  Here are some assumptions  for later use:
\begin{itemize}
\item [(1)]  $\fV_k $ is  connected and unibranch;  
\item[(2)]     $\fV$ is affine (for the sake of  considering closed formal 
subschemes in the non-noetherian case, see \ref{Preliminaries});

\item[(3)]     $\fV$ is    flat over $\Fo$;
\item[(4)]   $\fV_k $      contains no embedded points.
\end{itemize}
\end{asmp}   

\begin{lem}\label
{uniirred} 
  (1) Under Assumption \ref{basicasmp} (1), $\fV_{i,k}$ is  unibranch and irreducible.

(2) Under Assumption \ref{basicasmp} (4),  $\fV_{i,k}$ has no  embedded points.

       \end{lem}
       \begin{proof} 
(1) Since $f_{j,i} $ is \'etale,  $\fV_{i,k}$        is unibranch by definition.
By Lemma \ref{0e20},
the spectra of local rings of $\fV_{i,k}$ are irreducible. So $\fV_{i,k}$ is irreducible. 

(2) Since a flat   ring homomorphism preserves non-zero divisors, a flat morphism  for noetherian 
schemes maps an associated point  to an associated point. By \cite[Lemma 05AL]{Sta}, which says that embedded points of noetherian schemes  are non-generic associated points, (2) follows. (Or one may prove (2) by interpreting Assumption \ref{basicasmp} (4) as Serre's condition $S_1$, which is preserved by an \etale morphism.)
\end{proof}   

\begin{lem}\label{basic} 
Under Assumption \ref{basicasmp}, for    a closed formal subscheme $\fY\subset \wt \fV$   and  $y\in \fY(k)$, if
$\fY_y= \wt \fV_{y}$,  then $\fY= \wt \fV$.
\end{lem}  

\begin{proof}  
Claim: the natural map    $\cO_{\wt \fV}(\wt \fV)\to  \wh\cO_{\wt\fV,y}$   is injective.  Here $\wh\cO_{\wt\fV,y}$ is the  complete local ring at $y$.
Thus the composition of the diagram of  natural morphisms $$\cO_{\wt \fV}(\wt \fV)\to \cO_{\fY}(\fY)\to  \wh\cO_{\fY,y} \cong \wh\cO_{\wt\fV,y}$$
is injective (the condition in the lemma gives the last isomorphism).   Then   the first morphism in this   diagram  is injective (also surjective by definition), and thus an isomorphism.
So   $\fY= \wt \fV $.

We prove the claim. Since $\cO_{ \fV}( \fV) $ 
is the $\varpi$-adic 
completion of $\vil_i\cO_{\fV_i}(\fV_i) $ and  $\wh\cO_{\fV,x}$ is $\varpi$-adically complete (see  \cite[Lemma 090T]{Sta}),  
 we  only need to show that   the natural map  
\begin{equation}
       \label{inji}
      \cO_{\fV_i}(\fV_i) /\varpi^n\to \wh\cO_{\wt\fV,y}  /\varpi^n
       \end{equation}
          is injective  for every positive integer $n$.  
        Recall that  we have an isomorphism $\wh\cO_{\wt\fV,y} \cong \wh\cO_{\fV_i,\pi_i(y)}$ induced by $\pi_i$ as in the case $i=0$ discussed above. 
Then
the case $i=0$ is proved in Lemma \ref{basic0}.
For  the general $i$, the proof is the same.
 \end{proof}

Recall the  action   of  $ \fM $ on $\fI_{0,{    y}}$ defined in \ref{Toric action}.    Here and below, we understand  an $E^\circ$-point $P$ of $\fM_{\Fo}$ as a closed   formal $\Fo$-subscheme of $ \fM_{\Fo}$, i.e. its schematic image. 
\begin{cor}\label{samestab}     Under Assumption \ref{basicasmp},    for  a finite extension $E/F$  and  $y\in \fY(k)$, $ P\in\fM _{\Fo}(E^\circ)$ stabilizes  
$ \wt \fV_y $
if and only if  it stabilizes   $ \wt \fV$.             
\end{cor} 
 
    \begin{proof}        
Then $P$ stabilizes   $ \wt \fV$, by definition, if and only if the schematic preimage of $\wt  \fV$  under 
   the   multiplication morphism  $P\times \wt  \fV\to \fI_{0,\Fo}$ is $P\times \wt  \fV$; 
   $P$ stabilizes   $ \wt \fV_y$ if and only if the schematic preimage of $\wt  \fV_y$  under 
   the   multiplication morphism  $P\times \wt  \fV_y\to \fI_{0,\Fo,y}$ is $P\times \wt  \fV_y$. 
The corollary follows from Lemma \ref{basic}, applied to $P\times \wt  \fV \cong \wt \fV_{E^\circ}$.
  \end{proof}

\subsection{Linearity over $k$}\label{Linearity over}  
            We need some preliminaries.

Replacing $F$ by $ k[[t]]$ in the last subsection (so that $\fS_{k[[t]]}$ is defined via the natural  morphism $k\to k[[t]]$ and so on), we have the analogous discussion. In particular, as a special case of  the analog of Corollary \ref {samestab}, 
we have the following lemma. (Indeed, we may replace $k[[t]]$ by $\ol {k((t))}^\circ$ in the lemma. But we do not need it.)
   \begin{lem}\label{samestab1} 
    Under Assumption \ref{basicasmp} with (3) replaced by $\fV=\fV_k$,
    a point in    $ \fM_{k}\lb k[[t]]\rb$ stabilizes  
$ \wt \fV_y $ for some $y\in \fY(k)$ 
if and only if  it stabilizes   $ \wt \fV$.  
\end{lem} 
We need the following simple lemma.  
  \begin{lem}\label{as in1}
  
 Let $f\in k[[x_1,...,x_n]]$ and $f\neq 0$.  There exists a continuous   homomorphism  $k[[x_1,...,x_n]]\to   k[[t]]$ such that  the image of $f$ is nonzero.

  \end{lem}
  \begin{proof}  Let  $g$ be the sum of the  (finitely many) monomials in $f$ with  the minimal total degrees $m$. 
Consider the continuous   homomorphism  $k[[x_1,...,x_n]]\to  k[[t]]$ with   $x_i,i=1,...,n$ sent to $a_it$ where $a_i\in k$ is to be determined.
Then $g$ is sent to  $h(a_1,...,a_n)t^m$. Choose $a_i$'s such that  $h(a_1,...,a_n)\neq0$ and the lemma follows.
            \end{proof}

Now we   discuss linearity over $k$.
We call a subscheme  $Z$ of $  \fS_k$  (resp. formal subscheme  $\fV$ of $\fS_{\Fo}$)   linear at $x\in Z(k)$ (resp. $\fV(k)$)  if   
its formal completion at $x$ is 
a   formal subtorus of  $\fS_{k,x} $  (resp. $\fS_{x,\Fo}$).  
It is called linear if it is linear everywhere.

We will extensively use  two facts:  first,  for $y\in \fY(k)$, $\pi_0|_{\fI_{0,{    y}}}:\fI_{0,{    y}}\cong \fS_{\pi_0(x)}$ is an      isomorphism of formal tori (see \eqref{bcy} and the discussion below it);
second, the  action   of  $ \fM $ on $\fI_{0,{    y}}$  becomes multiplication   after a natural isomorphism
$ \fI_{0,{    y}}\cong \fM $, see Lemma \ref{Via}.  

\begin{prop}\label{-11}  Under Assumption \ref{basicasmp} with (3) replaced by $\fV=\fV_k$, if $\fV$ is  linear    at some  $x\in \fV(k)$, then it is    linear. 
\end{prop}
\begin{proof}

Claim: for  every $z\in   \fV (k)$, $  \fV_{z} $ is     stable  by all $ k[[t]]$-points of 
a closed formal subgroup $\fG$  of the formal torus  $\fS_{k,z}=\fS_{z,k}$  that is isomorphic to $\fV_x$.  
We will prove the claim later. Now we prove that $  \fV_{z} =\fG$ which gives the proposition.  
 Since $\fV$ is  linear at $x$, $\wh\cO_{\fV,x}$  is reduced. By the injectivity of $ \cO_{\fV,x}\to\wh\cO_{\fV,x} $ (Krull intersection theorem),
$\fV$ is reduced at $x$. Thus 
$\fV$  is  generically reduced by the openness of the reduced locus.
Then since $\fV$ has no embedded components, it is reduced. 
Then by the excellence   of $\fV$,  $  \fV_{z} $ is reduced.
Let $  \fC\subset\fV_{z} $ be a formal branch (see Definition \ref {fbran}) which is  integral by definition.   Since a $ k[[t]]$-point $P$ of 
$\fG$, understood as  the schematic image of $P$,  contains the  unit of $\fG$, the schematic image $\fC'$ of   the multiplication morphism $P\times  \fC\to  \fV_{z}$ contains $ \fC$. 
Since $R[[t]]$ is integral for an integral domain $R$,  
$\cO_{P\times  \fC}(P\times  \fC)$  is integral. So by  the definition of schematic image, 
$ \fC'$  is integral, and thus equals $ \fC$ by dimension reason.  
Since $ \fC$ contains the  unit of $\fS_z$,    $ \fC=\fC'$ contains  $P$, i.e.,  $ \fC$ contains $\fG\lb k[[t]]\rb$.
By Lemma \ref{as in1}, $\fG\lb k[[t]]\rb$ is schematically dense in $\fG$. So 
$  \fC$ contains $\fG$. 
Since $\fV$ is integral    and of finite type over $k$, all complete local rings have the same dimension.
So $\fC=\fG$.  Thus $  \fV_{z} =\fG$.

Now we prove the claim. Let  $y\in  \pi_0^{-1}  (\{x\})\cap \wt \fV$. 
By
$\wh\cO_{\wt\fV,y} \cong \wh\cO_{\fV,x}$  (see  \eqref{indisom}) and 
$\pi_0|_{\fI_{0,{    y}}}:\fI_{0,{    y}}\cong \fS_{x}$   of formal tori, 
$ \wt \fV_{y}$
is   stable  under  the multiplication on $\fI_{0,y,k}$  by   all $k[[t]]$-points of a closed formal subgroup of $\fI_{0,y,k}$  isomorphic to $\fV_x$.
By Lemma \ref{Via},    $ \wt \fV_y$
is   stable, now under the action of $\fM_{k}$, by  all $k[[t]]$-points of  a closed formal subgroup of $\fM_{k}$   isomorphic to $\fV_x$.  By   the ``only if" part of Lemma \ref{samestab1},    $ \wt \fV$
is   stable  by   all $k[[t]]$-points of a closed formal subgroup of $\fM_{k}$  isomorphic to $\fV_x$.  
Now we reverse the  reasoning for  $w\in  \pi_0^{-1}  (\{z\})\cap \wt \fV$.
By   the ``if" part of Lemma \ref{samestab1},    $ \wt \fV_w$
is   stable  by   all $k[[t]]$-points of a closed formal subgroup of $\fM_{k}$  isomorphic to $\fV_x$.   
By Lemma \ref{Via},    $ \wt \fV_w$
is   stable, now under the multiplication on $\fI_{0,w,k}$  by  all $k[[t]]$-points of  a closed formal subgroup of $\fI_{0,w,k}$  isomorphic to $\fV_x$. 
By \eqref{indisom}  with $x,y$ replaced by $z,w$ which gives $ \wt\fV_w  \cong  \fV_z$,
and 
$\pi_0|_{\fI_{0,{    w}}}:\fI_{0,{    w}}\cong \fS_{z}$   of formal tori, 
the claim follows.
\end{proof}

\begin{rmk}\label{-11m}  
(1) Proposition \ref{-11} is a  result of Chai \cite[(5.3)]{Chai}.

(2) Clearly, one may use the same idea to deal with linearity over $\Fo$. Indeed,  then one will get Corollary \ref{Chapc} below. However, we will give another proof. 


\end{rmk}

\subsection{Lift linearity }
The following result of Chai is important for us.
\begin{lem} [{\cite[(5.5)]{Chai}}]\label{Chap}   Let $Z\subset  \fS_k$ be a linear subsheme.
There is  a  unique  linear formal subscheme
$Z^\can$ of    $\fS_{\Fo}$ 
lifting  $Z$.
\end{lem}
\begin{rmk}\label{locall}Note that the endomorphism ring of $\wh\BG_m^n$, over  $\Fo$ or $k$, is $\BZ_p$. (One can see this easily by identifying  $\wh\BG_m^n$ with its  $p$-divisible group.)
Then one easily obtains a  classification of subtori. In particular, the  linear lifting of a formal subtorus of $\wh\BG_{m,k}^n$  to $\wh\BG_{m,\Fo}^n$ exists and is unique.
Lemma \ref{Chap}  is the globalization this fact.           
\end{rmk}

\begin{cor}\label{Chapc}  Under Assumption \ref{basicasmp}, if $\fV$ is  linear    at   $x\in \fV(k)$, then it is    linear. 

\end{cor} 

    \begin{proof}        
          We can write $\fV_{k}$ as a union of affine opens such that each affine open contains $x$. Thus we may assume that $\fV $ is affine. 
      Let    $Z$ be the underlying reduced  subscheme of $\fV_{k}$, which is linear 
by Proposition \ref{-11}.  Let $  Z ^\can $  be as  in Lemma \ref{Chap}. By   Remark \ref {locall}, 
$  Z ^\can _x=\fV_x$. By Lemma \ref{basic0}, $  Z ^\can =\fV$.
  \end{proof}
\begin{rmk}        
    Note that  Corollary \ref{Chapc} is a special case of  Theorem \ref{wwcor1}.       \end{rmk}






 




Our following proof of Theorem \ref{wwcor} is  based on the same idea as the proof of Corollary \ref{Chapc}, but worked out on $\fI_0$ rather than on $\fS$.
\begin{proof}[Proof of Theorem \ref{wwcor}] 
We may assume that $\fV $ is affine. 
We may enlarge $F$ so that the given translated   formal subtorus  
in 
$\fV_x$  are translations by $\Fo$-points. 

(1) Let    $Z$ be the underlying reduced subscheme of $\fV_{k}$, which is linear 
by Proposition \ref{-11}.  Let $  Z ^\can $  be as  in Lemma \ref{Chap}.
By    Remark \ref {locall}, $  \fV_x$  contains $P+\wt Z^\can_x$ for  some  $P\in \fS_x(\Fo)$. 
Let $\wt Z^\can$  be the restriction of $ \pi_0^{-1}(Z ^\can) $ to $\wt \fV_k$.  
By  \eqref{indisom}, we have  $\wh\cO_{\wt\fV,y} \cong \wh\cO_{\fV,x}$ and analogously $\wh\cO_{\wt Z^\can,y} \cong \wh\cO_{Z^\can,x}$. Then since
$\pi_0|_{\fI_{0,{    y}}}:\fI_{0,{    y}}\cong \fS_{x}$  is an isomorphism of formal tori, 
by     Lemma \ref{Via} (about  the action   of  $ \fM $ on $\fI_{0}$),  $ \wt \fV_y$  contains $h( \wt Z^\can_y)$ for  some  $h\in \fM(\Fo)$. 
Applying Lemma \ref{basic}  with  $\wt\fV$ replaced by  $   \wt Z^\can$ and $\fY$ replaced by $h^{-1}\wt\fV\cap Z^\can$, 
we have $ \wt Z^\can \subset  h^{-1}\wt \fV$ i.e., $h \wt Z^\can \subset  \wt \fV$ (here we used that $h$ is an $\Fo$-point but not for some field extension, see the convention above Corollary \ref{samestab}).      Let  $w\in  \pi_0^{-1}  (\{z\})\cap \wt \fV$.
Then (1) follows by  reversing the  reasoning, i.e. applying Lemma \ref{Via},
 the     isomorphism  $\pi_0|_{\fI_{0,{    w}}}:\fI_{0,{    w}}\cong \fS_{z}$ of formal tori and \eqref{indisom} with $x,y$ replaced by $z,w$  which gives $ \wt\fV_w  \cong  \fV_z$.

(2) 
The proof is almost the same as (1).  
Instead of    ``$ \wt \fV_y$  contains $h( \wt Z^\can_y)$ for  some  $h\in \fM(\Fo)$" 
now
$ \wt \fV_y$ is the union of $h_i \wt Z^\can_y$'s  where $i$ lies in a finite index set and   $h_i\in \fM(\Fo)$. 
So $\wt\fV$  contains $\cup_i h_i \wt Z^\can$.   
By Lemma \ref{basic} (with $\fY=\wt\fV\cap(\cup_i h_i \wt Z^\can)\subset \wt\fV$, here we use the condition that $\fV_k$ has no embedded points),   $\wt\fV$ is  contained in $\cup_i h_i \wt Z^\can$.    So  $\wt\fV=\cup_i h_i \wt Z^\can$. Then (2) follows.
\end{proof}

\subsection{Technical results}

Now we want to prepare a  result (Corollary \ref{smallercor}) that will only be used in  \ref{thm-5ps} (so that the reader  may skip  this subsection for the moment). 

\begin{lem}\label{tech0}
Under Assumption \ref{basicasmp}
(1)(2), further assume  that $\fV_k$ is reduced.
Then    for    a closed formal subscheme $\fY\subsetneq \wt \fV$,  there
 exists a positive integer $s$,    a closed  formal subscheme $ \fY_1\subset  \fY$ 
  and  a closed   subscheme $ Y\subset  \fV_k$ 
of $\dim Y<\dim \fV_k$ such that 
$ \fY\subset \fY_1\cup   \wt\fV/\varpi^s$ and 
$\fY_{1,k} \subset \pi_0  ^{-1}(Y)$. In particular,   for $w\in \fY(k)\bsl \pi_0  ^{-1}(Y)$,   $ \fY_w\subset   \wt \fV_w/\varpi^s$.        \end{lem}

         \begin{proof}
         Let $I\subset \cO_{\wt \fV}(\wt \fV)$ be the defining ideal of $\fY $. Let $s$ be the  maximal non-negative integer such that $I\subset    \varpi ^s \cO_{\wt \fV}(\wt \fV) $ (whose existence is guaranteed  by the $\varpi$-adic separatedness of $\cO_{\wt \fV}(\wt \fV)$).
Let
      $
      J:=\{a\in \cO_{\wt \fV}(\wt \fV):\varpi^sa \in I\}.
      $
      Then $I=\varpi^s J$ and thus $J\not\subset\varpi \cO_{\wt \fV}(\wt \fV)$. 
        Then $J/\varpi$ contains a nonzero element of $\cO_{\wt \fV}(\wt \fV)/\varpi=\vil_i \cO_{\fV_i}(\fV_i) /\varpi$. 
So this   nonzero element comes from  some $\cO_{\fV_i}(\fV_i) /\varpi$, and defines a closed subscheme $Y'\subsetneq \fV_{i,k}$. Since $ \fV_{i,k}$ is irreducible by Lemma \ref{uniirred} and reduced of the same dimension with $\fV_{k}$ (as an \etale morphism preserves reducedness and dimension), $\dim Y'<\dim \fV_{i,k}=\dim \fV_{k}$.
Then we take $Y=f_{i,0}(Y').$         
   \end{proof}

           \begin{lem}\label{tech}Under Assumption \ref{basicasmp},  further
           assume that there is a connected affine flat formal $\Fo$-scheme $\fV'$ with reduced unibranch special fiber $\fV'_k$ and 
            a finite    schematically surjective morphism $f:\fV'\to \fV$.  
                             For    a closed formal subscheme $\fY\subsetneq \wt \fV$, 
                              there
exists a positive integer $s$ 
  and  a closed   subscheme $ Y\subset  \fV_k$ 
of $\dim Y<\dim \fV_k$ such that  
for $w\in \fY(k)\bsl \pi_0  ^{-1}(Y)$ and $z=\pi_0(w)$, 
$f^{-1}\lb \pi_0|_{\fI_{0,{    w}}}(\fY_w)\rb\subset    \fV'_{f^{-1}(z)}/\varpi^s$.    \end{lem}
             \begin{proof}
             

             Let $\wt \fV'$ be the inverse limit of $\fV_i\times_\fV \fV'$'s. Explicitly, $
\cO_{\wt \fV'}(\wt \fV') $  is the $\varpi$-adic completion of  $ \vil_i \cO_{\fV_i}(\fV_i) \otimes _{\cO_{ \fV}( \fV)} \cO_{\fV'}(\fV') $.
By  Lemma \ref{-1lem} (2),   we have         $ 
\cO_{\wt \fV'}(\wt \fV') \cong \cO_{\wt \fV}(\wt \fV)\otimes _{\cO_{ \fV}( \fV)} \cO_{\fV'}(\fV') .
$
Let $\wt f:\wt \fV'\to \wt \fV$ be the natural morphism.   For $y\in \fY(k)$ and $x=\pi_0(y)$,
        by  Lemma \ref{-1lem} (2) again (for the first and third isomorphisms in \eqref{pddelt0})  and \eqref{indisom}    (for the second isomorphism),   
  \begin{equation}           \label{pddelt0}
\wh\cO_{ \wt\fV',\wt f^{-1}(y)}\cong \wh\cO_{ \wt \fV,y} \otimes _{\cO_{ \fV}( \fV)} \cO_{\fV'}(\fV')   \cong  \wh\cO_{  \fV,x}  \otimes _{\cO_{ \fV}( \fV)} \cO_{\fV'}(\fV') \cong\wh  \cO_{ \fV', f^{-1}(x)}.
      \end{equation}
     Indeed,  since the ideal defining $y$ contains $\varpi$,    the ideal it generates in $\cO_{\wt \fV'}(\wt \fV') $  contains $\varpi$, and it thus is open and closed. This is the closed ideal defining $\wt f^{-1}(y)$. So Lemma \ref{-1lem} (2) can be applied to  get the first isomorphism.

    Let $\fY'=\wt f^{-1}(\fY)\subset \wt \fV'$. Claim 1: $\fY'\neq \wt \fV'$.  The claim will be  proved later.
Since, by definition, 
$
\cO_{\wt \fV'}(\wt \fV')/\varpi$ is the direct limit of $ \cO_{\fV_i}(\fV_i)/\varpi \otimes _{\cO_{ \fV}( \fV)/\varpi} \cO_{\fV'}(\fV') /\varpi$s, which are finite \etale over $\cO_{\fV'}(\fV') /\varpi$,
we can apply the proof of  Lemma \ref{tech0} to  $\fY'\subsetneq \wt\fV'$ (here we need the unibranchness of 
$\fV'_k$) to find  a positive integer $s$, 
and a closed subscheme $
Y'\subset\fV'_k$
of  $\dim Y'<\dim \fV'_k$  such that  for $y'\in \fY'(k)$,
$ \fY'_{y'}\subset    \wt\fV'_{y'}/\varpi^{s}$  
unless   $\pi_0(y')\in Y'$. 
By Lemma \ref{samedim} (2) and the flatness, $\dim \fV_k=\dim \fV-1=\dim \fV'-1=\dim \fV'_k$.
Claim 2: for $y\in \fY(k)$ and $x=\pi_0(y)$, 
 \begin{equation}  \fY'_{\wt f^{-1}(y)} \cong f^{-1}\lb \pi_0|_{\fI_{0,{    y}}}(\fY_y)\rb.
 \label{6diag0}
\end{equation}
under \eqref{pddelt0}.
Then the lemma follows
with $Y=f(Y')$ whose dimension is $ \dim Y'<\dim \fV_k$.

To prove the claims, we use the following commutative diagram, where  the two isomorphisms are \eqref{indisom}   and     \eqref{pddelt0}, and
the others are the  natural ones:
 \begin{equation*}  
\xymatrix{
	 \cO_{\wt \fV}( \wt\fV)  \ar[r]^{ } \ar[d]^{ } &  \wt\cO_{\wt\fV,y} \ar[r]^{\cong}\ar[d]^{}  &    \wt\cO_{\fV,x}  \ar[d]^{ } \\
\cO_{\wt\fV'}(\wt\fV') \ar[r]^{  }    &	\cO_{ \wt\fV',\wt f^{-1}(y)} \ar[r]^{    \cong}    & \cO_{ \fV', f^{-1}(x)} }. 
\end{equation*}
Since $ \cO_{ \fV}( \fV) \to\cO_{\fV'}(\fV')$  is injective by the definition of schematic surjectivity, the
last vertical  morphism is injective by the Artin-Rees lemma. So  the middle vertical  morphism is injective.

We first prove Claim 1. 
Let $I\subset  \cO_{ \wt\fV}( \wt\fV)$ be the ideal defining $\fY$.
By Lemma \ref{-1lem} (1) and Lemma \ref{basic}, 
the image of $I$ in $\wh\cO_{\wt\fV,y}$ is nonzero. Thus the image of $I$  in $\cO_{ \wt\fV',\wt f^{-1}(y)} $ is nonzero. 
So the image of $I$ in $\cO_{\wt\fV'}(\wt\fV')$ is nonzero. Claim 1 is proved.

Now we  prove Claim 2. Indeed, by Lemma \ref{-1lem} (1) (and the definition of pullback), the left (resp. right) hand  side  of \eqref{6diag0} is defined by  the closure  of $I\cO_{ \wt\fV',\wt f^{-1}(y)}  $ in   $\cO_{ \wt\fV',\wt f^{-1}(y)}  $ (resp. the closure  of $I\cO_{ \fV', f^{-1}(x)} $ in   $\cO_{ \fV', f^{-1}(x)} $), where    $\cO_{ \wt\fV',\wt f^{-1}(y)}  $ (resp. $\cO_{ \fV', f^{-1}(x)}) $ becomes an $\cO_{ \wt\fV}( \wt\fV)$-module through the lower left  (resp. upper right) part of the diagram.
   \end{proof}

\begin{rmk}\label{asmp4}  We may remove Assumption \ref{basicasmp} (4) and instead  assume $\fY_y\subsetneq \wt \fV_y$ form some $y\in \fY(k)$ in Lemma \ref{tech}. Then the conclusion still holds by the same proof, except  that Lemma \ref{basic} is not needed anymore.
\end{rmk}
        \begin{cor} \label{smallercor}
Let $f:\fV'\to \fV$ be as is in Lemma \ref  {tech}.
Assume that for some $x\in \fV(k)$, 
$$  \fV_x\cap \lb \fV_x+ \lb \fS_x [p^r]\bsl \fS_x [p^{r-1}]\rb\rb\neq  \fV_x$$
for some positive integer $r$. There exists a positive integer $s$       and  a closed   subscheme $ Y\subset  \fV_k$ 
of $\dim Y<\dim \fV_k$
such that for  every $z\in \fV(k)\bsl Y(k)$,  
$$
f^{-1}\lb \fV_z\cap \lb \fV_z+ \lb \fS_z [p^r]\bsl \fS_z [p^{r-1}]\rb\rb\rb\subset \fV'_{f^{-1}(z)}/\varpi^s.$$

\end{cor}
\begin{proof} Up to enlarging $\Fo$, we may assume that $F$ contains all $p^r$-th roots of unity.
Let  $\fM[p^r]\wt \fV$ be the union of $h\wt \fV$'s, $h\in\fM[p^r](\Fo)$. By  the assumption in the corollary, \eqref{indisom}, 
and the two facts
recalled above Proposition \ref{-11}, we can let   $\fY=\wt \fV\cap \fM[p^r]\wt \fV$ in 
Lemma \ref  {tech}.  With $s,Y$  as in  Lemma \ref  {tech},  by \eqref{indisom}   with $x,y$ replaced by $z,w$  which gives $\wh\cO_{\wt\fV,w} \cong \wh\cO_{\fV,z}$,
and the two facts recalled above Proposition \ref{-11} with $y$ replaced by $w$, the corollary follows.
\end{proof} 
\begin{rmk}\label{asmp44}  By  Remark \ref{asmp4}, Assumption \ref{basicasmp} (4)  is not needed.
    \end{rmk}    
\section{Frobenii and linearity}
\label{Frobenii and linearity}

We introduce   Frobenii on the Siegel and Igusa formal schemes.
Then, we relate linearity and  Shimura subvarieties,
and discuss a new notion
of linearity.
We finally prove Theorem \ref{wwcor1}.

\subsection{Lifts of  Frobenii}\label{Lifts of  Frobenii}

Let
$\Fr:\BX \to \BX $ be the canonical lifting of the endomorphism defined by the absolute  Frobenius. Simply, on $\wh\BX $, $\Fr$ is the $p$  power morphism, and on $  \BX_{\et }$, $\Fr$ is the identity  morphism. 
Then $\ker\Fr=\wh\BX[p]$, and  we have the isomorphism $\BX/\wh\BX[p]\cong \BX.$


Let  $\phi_\fS: \fS\to  \fS$ be the morphism given by the natural transformation $A\mapsto A/ \wh A[p] $
(on $R$-points for $R\in \Nilp_{\BZ_p}^\op$).
Then the restriction of $\phi_\fS $ to $\fS_{\BF_p}$ is the absolute Frobenius on $\fS_{\BF_p}$.

Let  $\phi_\fI: \fI\to  \fI$ be the morphism given by the natural transformation $(A,\vep)\mapsto (A/ \wh A[p],\vep')$ where $\vep':A/ \wh A[p]\cong \BX_R/\wh\BX_R[p]\cong \BX_R$ is the  induced isomorphism. 
Then  the restriction of $\phi_\fI $ to $\fI_{\BF_p}$ is the absolute Frobenius on $\fI_{\BF_p}$.

Similarly, we have $\phi_{\fI_0}: \fI_0\to  \fI_0$ that lifts  the absolute Frobenius on $\fI_{\BF_p}$.
Then  the  natural projections 
$\fI \xrightarrow{\pi} \fI_0 \xrightarrow{\pi_0} \fS$
are compatible with the liftings of the Frobenii in an obvious way.
 
  Below, we will often use base changes of  the Frobenii $\phi_{\fS}$, $\phi_{\fI}$,  and $\phi_{\fI_0}$. If there is no confusion, we will use $\phi_{\fS}$, $\phi_{\fI}$,  and $\phi_{\fI_0}$ to denote their base changes  to lighten the notations.  (We also recall that  the base change  of the absolute Frobenius on a  ${\BF_p}$-scheme is the relative Frobeniu.)
  
  For $ x\in \fS({k})$,  
$ \phi_{\fS}|_{\fS_{x}}$ factors through $    \fS_{\phi_{\fS}(x)}$, and we always understand $ \phi_{\fS}|_{\fS_{x}}$ as from
$ {\fS_{x}}$ to  $    \fS_{\phi_{\fS}(x)}$. 
Similarly, For $ y\in \fI_0({k})$, 
$ \phi_{\fI_0}|_{\fI_{0,y}}$ is to $  \fI_{0,\phi_{\fI_0}(y)}$.
  \begin{lem}[{\cite[LEMMA 4.1.2 and  p 171]{Kat}}]\label{Kzz1}
(1) For $ y\in \fI_0({k})$,  $c_{\phi_{\fI_0}(y)}\circ\phi_{\fI_0}  \circ c_y^{-1}$ 
is the $p$-power endomorphism of the formal torus $\fM$.

(2) For $ x\in \fS({k})$  
and $ y\in \pi_0^{-1}(\{x\})$, $a_{\phi_{\fI_0}(y)}\circ\phi_{\fS}  \circ c_y^{-1}$ 
is the $p$-power endomorphism of the formal torus $\fM$.

  \end{lem}

For a positive integer $m$ and an open formal subscheme $\fU$ of $\fS_{F_0}$, let $\fU^{(m)}$ be the open formal subscheme 
of $\fS_{F_0}$ supported on 
the topological image of $\fU_k$ by  $\phi_{{\fS}}^m$. Then  $\phi_{{\fS}}^m|_{\fU} $ factors through $\fU^{(m)}$.
Below, we always understand $\phi_{{\fS}}^m|_{\fU} $ as from $\fU$ to $\fU^{(m)}$.
 \begin{cor}\label{vpr0} 
 Let $\fV$ be a closed formal  subscheme of an open formal subscheme $\fU$ of $\fS_{F_0}$.  
 
(1)  For a positive integer $m$, the formal completion 
of     $\phi_{{\fS}}^m|_{\fU}   (\fV)$ at $ \phi_{{\fS}}^{m} (x)$ is $ \phi_{{\fS}}^{m}|_{\fS_{x}}  (\fV_x)$. 

(2) If the formal completion  in (1)  has a formal branch (see Definition \ref {fbran})  that is 
a translated  formal subtorus, then $\fV_x$ has a formal branch that is 
a translated  formal subtorus whose schematic image by $ \phi_{{\fS}}^{m}|_{\fS_{x,\Fo}}$ is the former one. 

\end{cor}
\begin{proof}  

(1)
Let   $i:\fV\incl\fU$ 
(resp. $j:\fU_x\incl \fU$) be the natural morphism which is proper (resp.  flat).  
Apply Lemma \ref{Kaplem}  with $\fY=\fV$, $\fX={\fU}$,  $\fX'=\fU_x$,   $f=\phi_{{\fS}}^m|_{\fU}  \circ i$ and $g= j$ and note that $\phi_{{\fS},k}$ is a homeomorphism.

(2) By Lemma \ref{Kzz1} (2), the schematic preimage of a translated  formal subtorus by $ \phi_{{\fS}}^{m}|_{\fS_{x,\Fo}}$ is a union of translated  formal subtori of the same dimension. Then (2)
follows from (1) and 
Lemma \ref{samedim} (3).
\end{proof} 


\subsection{More on linearity}\label{Speciality}
\subsubsection{Linearity and speciality} 

The notion ``weakly linear" was used by Chai \cite[(5.3)]{Chai} over $k$. For our purpose, it should be understood as modeled on the (equivalent) alternative  ``weakly special" of ``geodesic" in the complex algebraic\footnote{As pointed out   by Gao,   \cite[Theorem 1.4]{UY} implies the equivalence without algebraicity.}
setting, see \cite{Moo}  for the  definition and equivalence.  And    weakly special subvarieties of $S_\BC$  are in fact  defined over  $\ol\BQ$. 

Let us
recall  two stronger notions. 
 First, for a formal subscheme $\fV$ of  $ \fS _{\Fo}$ and $x\in \fV(k)$, we call  $\fV$  quasi-linear  at  $x$  if  $\fV_x\subset \fS_{\Fo,x}$ 
is a   union of  translations by torsion points of formal subtori  of  $\fS_{\Fo,x}$  (following   Moonen \cite{Moo2}). 
We call  $\fV$  quasi-linear     if  $\fV_x$ is   quasi-linear   for every $x\in \fV(k)$. 
Second, a subvariety of $S_{\ol\BQ}$ is special if it is an
irreducible component of a  
Hecke translation of   a Shimura subvariety. 

It is natural to ask about  the relation between    linearity, which is an  analytic notion, and speciality, 
which is an    algebraic  notion.      
Assume   that   $\fV$ come from a closed 
subvariety $V$ of  $S_{\ol\BQ}$, i.e.,  $\fV$ is an open formal subscheme of the $p$-adic formal completion of the Zariski closure of $V$ in $S_{\Fo}$ if $V$ is defined over a subfield of $F$.
If $V$ is  weakly special, then $\fV$ is weakly linear.   This is  a result of Noot \cite{Noot}.
Conversely,  if for some $x\in \fV(k)$, $\fV_x$  is quasi-linear, then 
$V$ is    special.
This is  a result of  Moonen \cite{Moo2}.  
Two natural (expected) generalizations of this result would be: (1) replace  ``quasi-linear" by ``weakly linear"; (2)   the algebraicity of  quasi-linear/weakly linear formal subschemes.
See \cite{Chai} for a partial result about (2) on a product of Hilbert modular schemes.  
We hope to return to these generalizations  in the future.  

\subsubsection{Quasi-linearity} 
We have a ``quasi-linear" version of Conjecture \ref{ALconj}.
\begin{conj}\label{ALconj2}
Assume that $\fV$ is a closed formal  subscheme of an open formal subscheme $\fU$ of $\fS_\Fo$, $x\in \fV(k)$ and $\fT\subset \fS_{\Fo,x}$ a   formal subtorus  translated  by a torsion point. 
If $\fV$ is the  schematic image of $\fT\to \fU$,  then $\fV$ is   quasi-linear. 
\end{conj} 
Theorem \ref{wwcor1}  follows  from the following  theorem  on  Conjecture \ref{ALconj2}.

\begin{thm}  \label{wwcor2} 
Conjecture \ref{ALconj2} holds if $\fV_k$ is  unibranch without embedded points and $\fT$  contains a torsion point. 
\end{thm}
The following analog of Theorem \ref{wwcor} (2) for quasi-linearity is a special case of   Theorem \ref{wwcor2}.
\begin{cor}  \label{wqw1}

Let $\fV$ be a weakly linear  formal subscheme   of $\fS_\Fo$ such  that $\fV_k$ is unibranch and  has no embedded points.
If  $\fV$ is    quasi-linear   at  $x$, then    it  is  quasi-linear. 
\end{cor} 


Recall that we also have the notion ``linear" in \ref {Linearity over}. 
Now we discuss the relation between the notions ``linear",  ``quasi-linear" and ``weakly linear".

\begin{prop} \label{lql}
Let $\fV$ be a    closed formal  subscheme of an open formal subscheme $\fU$    of $\fS_\Fo$. Assume that $\fV$ is reduced and flat over $\Fo$, and $\fV_k$  has no embedded points. 
Assumes that   there exists
a positive integer  $m$  such that 
$\fV':=\phi_\fS|_{\fU}^m(\fV)$   is   linear,  then $\fV$ is  quasi-linear.

\end{prop} 
\begin{proof}

Let $Z$ be the underlying reduced subscheme of $\fV_{k}$.    We  show that  $Z$ is  linear   as follows.    
By the excellence and reducedness  of $Z$, for $x\in Z(k)$, 
$  Z_{x} $ is reduced. 
Since the relative Frobenius induces (a homeomorphism  on $\fS_k$ and)
homeomorphisms on the spectra of the completions of the local rings of $\fS_k$,  $  Z_{x} $   has the same number of   formal branches with $\fV'_{k,\phi_\fS(x)}$, which is 1 by  the linearity of $\fV'_k$ at $\phi_\fS(x)$.
Then by Lemma \ref{Kzz1} (2), the linearity of $\fV'_k$ at $\phi_\fS(x)$ and dimension reason (see Lemma \ref{samedim} (3)),
$Z$ is linear.

We may assume that  $\fV$  is connected.
As in the  proof of Corollary \ref{Chapc}, we may assume that $\fV $ is affine.
By Lemma \ref{Kaplem} and Remark \ref   {locall}, we may enlarge $F$ and assume that $F$ contains all $p^m$-th roots of unity.
Now we can apply the discussion in \ref{Consequences}. We have $\wt\fV$ and the analogs  $\wt\fV'$ and $\wt Z^\can$ in $\fI_{0,\Fo}$ (where  $  Z ^\can $  is as  in Lemma \ref{Chap}) which we can choose so that $\wt\fV, \wt Z^\can\subset \phi_{\fI_0}^{-m}(\wt\fV')$,
since  the  natural projections 
$ {\pi_0}  $
is compatible with the liftings of the Frobenii.
By  Lemma \ref{Via} and Lemma \ref{Kzz1} (1),
$\phi_{\fI_0}^{-m}(\wt\fV')$ is  the   union of $h \wt Z ^\can$, $h\in\fM[p^r]$.
Claim: $\wt\fV$ is  the   union of some $h \wt Z ^\can$'s.
Then the lemma follows from Lemma \ref{Via} (about  the action   of  $ \fM $ on $\fI_{0}$),  $\wh\cO_{\wt\fV,y} \cong \wh\cO_{\fV,x}$ (see  \eqref{indisom}) and 
$\pi_0|_{\fI_{0,{    y}}}:\fI_{0,{    y}}\cong \fS_{x}$   of formal tori for all $x\in \fV(k)$ and $y\in  \pi_0^{-1}  (\{x\})\cap \wt \fV$. 

To prove the claim, instead of the the non-noetherian formal schemes, we work on their completions which are noetheiran. 
Take $x\in \fV(k)$ and $y\in  \pi_0^{-1}  (\{x\})\cap \wt \fV$.  Then
the excellence  of $\fV$,
$\wt\fV_y \cong \fV_x $ (see \eqref{indisom}) is reduced. Thus    $\wt\fV_y$ is  the union of  its formal branches. And  the formal branches are flat over $\Fo$ (by the flatness of $\wt\fV_y  $ and that $\Spf \Fo$ is integral and regular of dimension 1, see \cite[III.9.7]{Har}).
For  $h\in\fM[p^r]$, $\fZ_h=(\wt\fV\cap h \wt Z ^\can)_y$ is a union of some formal branches of $\wt\fV_y$.
We only need to prove that $\fZ_h\neq \emptyset$, $\fZ_h= h \wt Z ^\can_y$. (Then 
by Lemma \ref{basic}, $\wt\fV $ is  the union of some $h \wt Z ^\can_y$'s, and thus quasi-linear.) 
   Assume that $\fZ_h\neq h \wt Z ^\can_y$,  by 
   the flatness, $\fZ_{h,k} \neq (h\wt Z ^\can_{y})_k=Z ^\can_{k,y}$.  
   Then  $(  \wt\fV\cap  h\wt Z ^\can)_k \neq \wt Z ^\can_k $.        
Since $\wt Z_k$ is the inverse limit of reduced irreducible schemes (see \ref{Consequences}), the image of 
$(  \wt\fV\cap  h\wt Z ^\can)_k$ in $  Z_k$ is a strictly smaller closed subscheme. An irreducible component of this 
strictly smaller closed subscheme gives an embedded component of $\fV_k$,  contradiction.
\end{proof} 
 \begin{rmk}(1) The  proof indeed gives us stronger linearity of $\fV$. First,
$\fV_x\subset \fS_{\Fo,x}$ 
is a   union of  translations by a set $T_x$ torsion points  of a single formal subtorus  of  $\fS_{\Fo,x}$. Moreover, for $x,x'$, there exists an isomorphism $\fS_{\Fo,x}\cong \fS_{\Fo,x'}$ such that
$T_x\cong T_{x'}$ under this isomorphism. 

(2) In Theorem \ref{wwcor2}, Proposition \ref{lql}, Corollary \ref{wqw1} and Corollary \ref{wqw}, we also have 
this stronger linearity of $\fV$ of (1).
\end{rmk}

The following corollary is an analog   of the fact that a weakly special subvariety   is special if and only if it contains a special (i.e., CM) point.
It will not to be used later.
Its proof   is    close to the proof of Theorem \ref{wwcor2} (in the end of this section), and thus omitted. 
\begin{cor}  \label{wqw}

Let $\fV$ be a weakly linear  reduced formal subscheme   of $\fS_\Fo$ such  that $\fV_k$ is unibranch  and has no embedded points.
If for some $x\in \fV(k)$, $\fV_x$ contains a torsion point of $\fS_{\Fo,x}$,  then $\fV$ is  quasi-linear. 
\end{cor} 
\begin{rmk}If  the unibranch  assumption in  Proposition \ref{-11} is removable,  so is the one  in Theorem \ref{wwcor2}, Proposition \ref{lql}, Corollary \ref{wqw1} and Corollary \ref{wqw}.
\end{rmk}

\subsection{Proof of Theorem \ref{wwcor1}} We need an extra lemma.
\begin{lem}\label{djg} Let $\fZ\subset \wh\BG_{m,\Fo}^n$ be a reduced closed formal subscheme that is flat over $\Fo$ and stabilized by the $p^c$-power endomorphism of $ \wh\BG_{m,\Fo}^n$, which we denote by $p^c$ for short,  for some positive integer $c$. 
Then $ \fZ$ contains  a union of formal subtori of $\wh\BG_{m,\Fo}^n$ that contains $  \fZ_k'$,  the underlying reduced formal subscheme of $\fZ_{k}$.   
\end{lem}
\begin{proof}  
First, we find  a formal subtorus in $\fZ$.
Since
$p^{c}$ induces the relative Frobenius on $\Spec \cO_{\fZ_k}(\fZ_k)$ which is a homeomorphism, it permutes the  branches    of $\fZ_k$. Then a branch $\fC$ of $\fZ_k$ is  stabilized
by  $p^{cc_1} $ for  some positive integer $c_1$. 
Let  $\wt\fC$ be a formal branch  of $\fZ$ such that $\fC\subset \wt\fC_k$. 
The (finite) union of the schematic images of $\wt\fC$ by non-negative powers of $p^{cc_1} $ is 
stabilized
by  $p^{cc_1} $.  Then one of them, say  $\fB$,  satisfies 
$
p^{cc_1c_2}( \fB)\subset \fB$ for a positive integer  $c_2$.
A result of de Jong \cite{dJ} implies 
that  $\fB$ is  a formal subtorus.   Note that  $\fC\subset\fB_k$.
The lemma follows.
\end{proof} 
Now we prove  Theorem \ref{wwcor2}, and thus Theorem \ref{wwcor1}.
\begin{proof}[Proof of Theorem \ref{wwcor2}]

Let $m$ be a positive integer such that  $\phi_\fS^m|_{\fS_{x,\Fo}}(\fT)$ is a formal subtorus of  $\fS_{\phi_\fS^m(x),\Fo}$ (which is possible by Lemma \ref{Kzz1}). 
Let $\fV'$ be the  schematic image of $\phi_\fS^m|_{\fS_{x,\Fo}} (\fT)\to\fU^{(m)}$ (as defined below Lemma \ref{Kzz1}).   Since $\fV\subset \phi_\fS^{-m}(\fV')$ and $ \phi_\fS|_\fU^m(\fV)\supset \fV'$ by the definition of schematic image, 
$ \phi_\fS|_\fU^m(\fV)= \fV'$.
Similarly, let $n$ be a positive integer such that
$\phi_{\fS }^n$ stabilizes
$ \phi_\fS^m(x)$ so that   $\phi_\fS^n|_{\fS_{ \phi_\fS^m(x),\Fo}} $ stabilizes  $\phi_\fS^m|_{\fS_{x,\Fo}} (\fT)$.
Then $\phi_{\fS }^n$ stabilizes $\fV'$.

Let $Z$ be the underlying reduced   subscheme  of    $\phi_\fS|_\fU^m(\fV_k)$,    equivalently, of $\fV_k'$    (see Lemma \ref{samedim} (1)). 
By the excellence  of $Z$, for $z\in Z(k)$, 
$  Z_{z} $ is reduced. 
Since the relative Frobenius induces a homeomorphism  on $\fS_k$ and
homeomorphisms on the spectra of the completions of the local rings of $\fS_k$,  $  Z_{z} $   has only one  formal branch, and thus is integral.
Claim: $Z$ is linear and $  Z ^\can=  \fV'$,  where $  Z ^\can $  is as  in Lemma \ref{Chap}.  
Then since  $\fV$ is reduced by the definition of schematic image and the reducedness of $\fT$, the  theorem follows from Proposition \ref{lql}.

Now we prove the claim. 
By the definition of schematic image, $\phi_\fS^m|_{\fS_{x,\Fo}} (\fT)$ is reduced, and then $\fV'$ is reduced. 
By the excellence    of $\fV'$,  $  \fV'_{z} $ is reduced for every $z\in \fV'(k)$.  For $z\in \fV'(k)$ stabilized by $\phi_\fS^{n'}$ where $n'$ is a positive integer multiple of $n$,  $  \fV'_{z} $ is  stabilized by $\phi_\fS^{n'}$. 
By Lemma \ref{djg} and that $  Z_{z} $  is integral, $Z$ is linear at $z$ and 
$  \fV'_{z} $ contains the unique formal subtorus of $\fS_{\Fo,z}$ whose reduction is $Z_z$ (see Remark \ref{locall}).
Let $  Z ^\can $  be as  in Lemma \ref{Chap}. Then  $  Z ^\can _z$ is this unique formal subtorus. So 
$  Z ^\can\subset  \fV'$.   
Since $\phi_\fS^m|_{\fS_{x,k}} (\fT_k)$ is contained in the formal completion of $ \phi_\fS|_\fU^m(\fV_k)$
at $\phi_\fS^m(x)$, it   is contained $Z_{\phi_\fS^m(x)}$, 
So $\phi_\fS^m|_{\fS_{x,\Fo}} (\fT)  \subset Z^\can_{\phi_\fS^m(x)}$.
So $   \fV'\subset Z ^\can $.      
\end{proof}

   \section{More notions}
   
 \label{More notions}

\subsection{Hecke action}\label{Hecke action} Let $\BA_f^p$ be the ring of finite-outside-$p$ adeles of $\BQ$. For a  principally polarized $g$-dimensional abelian
scheme   $A$ over a connected $\BZ_p$-scheme,  a similitude infinite prime-to-$p$ level structure  is  a  similitude symplectic  isomorphism  
\begin{equation*}
 \BA_f^{p,2g}\cong \vpl_{N,\ p\nmid N} A[N] \otimes \BQ\label{level}
 \end{equation*} 
fixed by the fundamental group of the $\BZ_p$-scheme (this condition is independent of the choice of the base point for defining the  fundamental group).
The functor    assigning to  $R\in \Nilp^\op_{\BZ_p} $     the set of prime-to-$p$ isogeny classes of principally polarized  ordinary  abelian $R$-schemes 
with  infinite level  structure is representable by  $$\wt \fS =\vpl_{N,\ p\nmid N}\fS_{N}.$$  It is equipped with a natural $\GSp_{2g}(\BA_f^p)$-action.
 Similarly, we have the infinite prime-to-$p$ level  version  $\wt \fI_0$  of $\fI_0$ with  a  $\GSp_{2g}(\BA_f^p)$-action. 
(We do not need the one of $\fI$). 
It is easy to check that   the natural
projection $\wt\pi_0:\wt\fI_0 \to \wt\fS$   is $\GSp_{2g}(\BA_f^p)$-equivariant.    
Moreover, we have the following commutative (in fact Cartesian) diagram
  \begin{equation}  
\xymatrix{
	 \wt\fI_0  \ar[r]^{\wt\pi_0} \ar[d]^{\pr_{\fI_0}}  &\wt\fS  \ar[d]^{\pr_{\fS}} \\
\fI_0 \ar[r]^{  \pi_0 }    &  \fS},\label{faca}
\end{equation}
where the vertical morphisms are natural projections.

The   action of the formal deformation torus $\fM$ on $\fI_{0}$, defined in   \ref {Toric action}, obviously lifts to an action of $\fM$ on $\wt\fI_{0}$. 
It is direct to check that the action of $\fM$ on $\wt\fI_{0}$ commutes with the $\GSp_{2g}(\BA_f^p)$-action, and 
the actions of $\fM$ on $\wt\fI_{0}$ and $\fI_{0}$ commute with the natural projection $\pr_{\fI_0}$.
Thus for $h\in \fM(W) $ and $T\in \GSp_{2g}(\BA_f^p)$, we have 
\begin{equation}  \label{33} 
\pi_0\circ h\circ \pr_{\fI_0}\circ T=  \pi_0\circ \pr_{\fI_0}\circ h\circ T=
\pr_{\fS}\circ     \wt \pi_0\circ T\circ h=\pr_{\fS}\circ     T\circ \wt \pi_0\circ h.
\end{equation}
(It can be displayed in a $3\times3$ commutative diagram, which is left to the reader.)
Replacing $W$ by $\Fo$, \eqref{33} still holds after base change.

For $z\in \wt  \fI_0({k})$ or $\wt \fS(k)$,  similar  to \eqref{bcy} and the discussion below it, we have a formal torus structure on the formal residue disc at $z$. It is directly to check that $ \pr_{\fS}, \pr_{\fI_0}$ (thus all morphisms in \eqref{faca}) are isomorphisms of formal tori when restricted to each formal residue torus. 

For $z\in \wt  \fI_0({k})$, $\wt \fS(k)$,  $  \fI_0({k})$ or $  \fS(k)$, we use $z^\can$ to denote the canonical lifting of $z$, which is by definition the unit of the formal residue torus at $z$.
Then all morphisms in \eqref{faca} commutes with taking canonical lifting.
It is also direct to check that  the $\GSp_{2g}(\BA_f^p)$-action commutes with taking canonical lifting.

Finally,   we define the notion ``Hecke--Frobenius orbit". We first define
``prime-to-$p$ Hecke orbit". 
For $x\in \fS({\Fo})$, $\pr_{\fS}^{-1}(\{x\})$ is a union of $\Fo$-points, since $ \pr_{\fS}$ is an isomorphism  when restricted to each formal residue disc. 
Define the prime-to-$p$ Hecke orbit of  $x$  in $\fS({\Fo})$ to be the  image by $\pr_{\fS} $ of the $\GSp_{2g}( \BA_f^{p})$-orbit of some $\Fo$-point $\wt x\in \pr_{\fS}^{-1}(\{x\})$ in $\wt  \fS(\Fo)$. Clearly, this definition does not depend on the choice of the lift $\wt x$.     
Then the Hecke--Frobenius orbit of $x$ is the union of the images  of  its prime-to-$p$ Hecke orbit by $\phi_{\fS}^n$'s, $n\in \BZ$. 
 The same discussion applies to $\pr_{\fI_0}$.
 
\subsection{Perfectoid Igusa space}\label{Lifting}
   Let  $\CP$ be the $p$-adic completion of $\ol F$ and $\Cflat$   the $t$-adic completion of an  algebraic closure  of $k((t))$.   Then $\CP$ and $\Cflat$ are perfectoid fields, and 
 $\Cflat $ is  a   tilt of $\CP$ in the sense of Scholze \cite{Sch12}. 
We have  the tilting equivalence \cite{Sch12}  between the  category of perfectoid spaces over $\Spa(\CP,\CPo)$ and $\Spa(\Cflat,\Cfcc)$.
The image of an object or a morphism over $\Spa(\CP,\CPo)$  under  the tilting equivalence
is called its tilt. 

Let $\cI/\Spa(\CP,\CPo)$ be the adic generic fiber of $\fI_{\CPo}$  in the sense of \cite{SW}, as in \cite{CS}. By Proposition \ref{CS3},  $\cI$ is a perfectoid space whose tilt $\cI^\flat$ over $\Spa(\Cflat,\Cfcc)$ is the adic generic fiber of $\fI_{\Cfcc}$.
Then we have the tilting bijection \eqref{tiltbij}:
$$\rho :\cI(\CP,\CPo)\cong \cI^\flat(\Cflat,\Cfcc).$$ 
Let $\phi_\cI$ be the endomorphism of $\cI$
induced by the base change  of $\phi_{\fI }$ to $ {\CPo}$.     Let $\phi_{\cI^\flat}$ be the endomorphism of $\cI^\flat$
induced by the base change of  $\phi_{\fI  }$ to $\Cfcc$ (along $\BZ_p\to \BF_p\to \Cfcc$).
Then 
similar to \eqref{2.3.2},   one can check that
    $\phi_{\cI}$ is the tilt of  $\phi_{\cI^\flat}$.
  By \cite[Lemma 2.3.1]{Qiu}, we have \begin{equation}\label{2.3.1}\rho\circ \phi_{\cI}=\phi_{\cI^\flat}\circ \rho.\end{equation}

For ${  x}\in  \fS(k)$, let $A_{  x^\can}\in \fS(W(k))$  corresponding to  $x^\can$.   Fix  an isomorphism $\vep_{  x}^\can:A_{  x^\can}[p^\infty]\cong \BX_{W(k)}$    with reduction $\vep_{  x} :A_{  x} [p^\infty]\cong \BX_k$.  
Then  $(A_{  x^\can},\vep_{  x}^\can)\in \fI(W(k))$ and $(A_{  x} ,\vep_{  x})\in \fI(k)$.
We still use $(A_{  x^\can},\vep_{  x}^\can)$ to denote the corresponding point in $\cI(\CP,\CPo)$ 
and use  $(A_{  x} ,\vep_{  x})$ to denote the corresponding point in $ \cI^\flat(\Cflat,\Cfcc)$.
Then by \cite[Theorem 5.2]{Sch12} (explicated in  \cite[Lemma 2.3.2]{Qiu}),  we have 
\begin{equation}\label{2.3.2} \rho ((A_{  x^\can},\vep_{  x}^\can))=(A_{  x} ,\vep_{  x}).
\end{equation}

               \subsection{$p$-adic distance}\label{distance} 
               
               Let $|\cdot|$ be the $p$-adic norm on $\oFo$.   
Let $\fX$ be a formal scheme over $ \Fo$.
 For a
formal subscheme   $\fZ\subset \fX$  
and  a point $x\in \fX(\oFo)$, define \begin{equation*}d(x,\fZ)=d_\fX(x,\fZ)=\inf \{|\delta|:  \delta\in \oFo\mbox{ such that }  (x\MOD \delta)\in \fZ(\oFo/\delta) \},
\end{equation*}
Here $\inf \emptyset:=\infty$ so that $d(x,\fZ)<\infty$ if and only if the reduction of $x$ is in $\fZ(k)$. 
If $\fZ $ is closed in $\fX$,    the distance can be defined using the ideal defining  $\fZ$, see for example \cite[2.2]{Qiu}.
Define
\begin{equation}
\fZ_\ep=\{x\in \fX(\oFo):  d(x,\fZ)\leq \ep \}.
\label{nbhd}
\end{equation}

Directly from the definition, we have the following lemmas. 
\begin{lem}\label{open}  Assume that $\fZ\subset \fU\subset \fX$ as formal schemes, then for $x\in \fU(\oFo)$, $d_\fU(x,\fZ)=d_\fX(x,\fZ)$. 
   \end{lem}

\begin{lem}\label{base} Let  $E/F$ be a finite extension. 
Then 
$d(x,\fZ_{E^\circ})= d(x,\fZ)$.

      \end{lem}  
\begin{lem}\label{sub} Let 
$\fZ'$ be a  formal subscheme of $\fZ$. Then 
$d(x,\fZ)\leq d(x,\fZ')$.

      \end{lem}

       \begin{lem}\label{union} 
Let $\fZ= \fU\cup  \fV$, where $\fU,\fV$ are closed formal subschemes of $\fZ$. Assume that    there exists $\delta\in \Fo-\{0\} $   vanishing on $\fV$.
Then  
$d(x,\fU)\leq d(x,\fZ) /|\delta|.$
      \end{lem}  
     \begin{lem}\label{internbhd}  
Let $\fZ= \fU\cap  \fV$, where $\fU,\fV$ are closed formal subschemes of $\fX$.  Then  
$\fU_\ep\cap \fV_\ep=\fZ_\ep$.

      \end{lem} 

     \begin{lem}\label{sub2}   
For a morphism $f:\fX\to \fY$ of formal scheme over $ \Fo$,  
and a  formal subscheme $\fW$ of $\fY$, we have 
$d_{\fY}(f(x),\fW)=d_\fX\lb x,f^{-1}(\fW)\rb.$
\end{lem}
          
Assume that  $f:\fX\to \fY$  is a   morphism of locally noetherian formal schemes,  so that the schematic image 
$f(\fZ)$ is defined.
Note that $\fZ$ is a   formal subscheme of   $f^{-1}(f(\fZ)).$ Then Lemma \ref{sub} and Lemma \ref{sub2} imply  the following corollary.
      \begin{cor}\label{image} We have
  $d(f(x),f(\fZ) )\leq d(x,\fZ).$
      \end{cor}
   


   \section{Bounded ramification}
\label{Bounded ramification}

We  prove Theorem \ref{jointthm} (1).
The reader may  refer to \ref{Structure and strategy} for an overview of the proof.

          \subsection{Canonical liftings on $\fI_0$}\label{unt}
                
\begin{prop}\label{lem1}    Let $\fW$ be  a    formal subscheme of $\fI_{0,\Fo}$  and $\{y_n\}_{n=1}^\infty$   a sequence in $\fW_{0}(k)$ such that     $d(y_n^\can,\fW)\to 0$.  Let
$Y\subset \fW_k$ be a  closed  subscheme 
such  that 
\begin{itemize}\item[(1)]  for every integer $N>0$, $Y$ is contained in  the Zariski closure of $\{y_n\}_{n=N}^\infty$;      \item [(2)]
there exists $y\in Y(k)$ and a positive integer $c$ such that   the formal completion  $Y_y $ is reduced and  satisfies $\phi_{\fI_0}^{c}\lb Y_y \rb\subset Y_y$ (in particular,  $\phi_{\fI_0}^{c}(y) =y$).
        \end{itemize}
Then    
$\fW_{y}$ contains  a union of formal subtori of $\fI_{0,y}$ that   contains $  Y_y$   as a closed formal subscheme. 
\end{prop}

\begin{proof} Let $\pi:\fI \to \fI_0 $  be the natural  morphism as in  \eqref{pi0}, and $\pi^\flat$ the base change of  $\pi$ to $\Cfcc$ (along $\BZ_p\to \BF_p\to \Cfcc$).
Still use $\pi$ to denote the corresponding morphism between $\cI$ and the adic generic fiber  of $\fI_0$. 
Since $ d(y_n^\can,\fW)\to 0$, by     \eqref{2.3.2} and Lemma \ref{techprop5} (with Assumption \ref{asmp5} locally assured by Proposition \ref{CS3}),                  $ \pi\lb  \lb\pi^\flat\circ\rho\rb^{-1}\lb Y({\Cfcc}) \rb  \rb  $ consists of 
adic generic fibers of points in
$ \fW(\CPo)$. 
Note that  taking adic generic fibers induces a bijection $  \fI(\Cfcc) \cong  \cI (\Cflat,\Cfcc) $ (see \cite[Lemma 2.1.3]{Qiu}).
Let     $Y_y^\sharp\subset  \fW_y({\CPo}) $ correspond to $\pi\lb \lb\pi^\flat\circ\rho\rb^{-1}\lb Y_y({\Cfcc}) \rb\rb$. 

Recall that $\pi ^\flat$ is compatible with   the  Frobenii and $\pi $ is compatible with the liftings of the Frobenii.
Then by condition (2) on $Y_y$ and  
\eqref{2.3.1} (which says $\rho$  is compatible with the   $ \phi_{\cI},\phi_{\cI^\flat}$),   we have
$\phi_{\fI_0 } ^{c}\lb Y_y ^\sharp\rb \subset   Y_y^\sharp .$ Let $\fY\subset \fI_{0,y,\Fo}$ be the  schematic closure of $Y_y^\sharp$. Then $ \fY$ is   reduced and  flat over $ \Fo$, and 
$\phi_{\fI_0 } ^{c}\lb \fY\rb \subset   \fY.$
Then by Lemma \ref  {djg}, 
it is enough to show that
$Y_y\subset \fY.$ 

Now we show that
$Y_y\subset \fY.$  By  the construction of  
   $\fY$, the image of the reduction map $Y_y(\Cfcc)\to Y_y(\Cfcc/t) $ is contained in the image of $  \fY(\CPo/\pi)= \fY(\Cfcc/t) $.
       We only need to show that the image of the reduction map $Y_y(\Cfcc)\to Y_y(\Cfcc/t) $ is schematically dense in $Y_y$, i.e., not  contained in  a  closed formal subscheme of $Y_y$  strictly smaller than $Y_y$.  
         Since $Y_y$ is reduced by assumption,   it is enough to show this for each formal branch of $Y_y$, and thus we may assume that $Y_y$ is integral. 
          By the Noether normalization lemma for complete local domains \cite[Lemma 032D]{Sta}, $Y_y=\Spf R'$ such that $R'$ is finite over a subring       $R\cong k[[x_1,...,x_n]]$.    We only need to show that the image of the reduction map $\Spf R(\Cfcc)\to \Spf R(\Cfcc/t) $   is schematically  dense in $ \Spf R$. This 
       To show this,    let $f\in R$ and $f\neq 0$.  There exists a continuous   homomorphism  $R\to \Cfcc$ such that  the image of $f$ is nonzero (see Lemma \ref{as in1}). Then for some large positive integer $m$, $\phi^{1/p^m}(f)\not\in t \Cfcc$.  Thus  the  $\Cfcc/t$-point $\phi^{1/p^m}/t$ is not in the zero locus of $f$. 
\end{proof} 
   \begin{rmk}
In the proof, we can 
not take schematic closure in $\fI_{0,\Fo}$ 
due to its non-noetherianness.
   \end{rmk}

\subsection{Unramified points on Siegel moduli schemes}\label{Unramified points on Siegel moduli schemes}
We first deal with prime-to-$p$ Hecke orbits.

    \begin{lem}\label{thm1'} Let $\fV$ be  a    formal subscheme of $\fS_{\Fo}$. 
Let  $\{ x_n\}_{n=1}^\infty$ be a sequence in the prime-to-$p$ Hecke orbit of   $x_1\in \fS(\Fo)$ such that the reduction $\ol{x_n}$ of $x_n$ is 
in  $\fV(k)$  and 
$d(x_n,\fV)\to 0$. Let $Z$ be   the Zariski closure of  $\{\ol {x_n}\}_{n=1}^\infty $   in $\fV_{k}$.  For $x\in Z(k)$,   $\fV_x$ contains a union of  translated formal subtori that contains   $Z_x$ as a closed formal subscheme.
\end{lem}

\begin{proof}
       Up to replacing  $\{ {x_n}\}_{n=1}^\infty$ by    a  subsequence, we may assume  that $Z$ is the Zariski closure of  $\{\ol {x_n}\}_{n=1}^\infty $. 
     We  want  apply Proposition \ref{lem1} to prove the lemma.  We have four steps. 
     
     First, we construct $\{y_n\}_{n=1}^\infty$  and $\fW$.
    Choose ${y_1}\in \fI_{0}(k)$ such that $\pi_0(y_1)=\ol{x_1}$.
      By    the isomorphism  $\pi_0|_{\fI_{0,{    y_1}}}:\fI_{0,{    y_1}}\cong \fS_{x_1}$ of formal tori  
      and the $\fM$-action on $\fI_{0,{    y_1}}$ described in  Lemma \ref{Via}, there exists   $h\in \fM(\Fo)$ such that $\pi_0\lb h{y_1^\can}\rb=x_1$.
                  Let $\wt y_1\in \pr_{\fI_0}^{-1}({y_1})$. Then  $\pr_{\fI_0} (\wt y_1^\can)= {y_1^\can}$, since  $\pr_{\fI_0} $ commutes with taking canonical lifting.
                  Let $\wt x_1=\wt\pi_0(h\wt y_1^\can)$. By \eqref{faca} and that the actions of $\fM$ on $\wt\fI_{0}$ and $\fI_{0}$ commute with the natural projection $\pr_{\fI_0}$, we have 
$$\pr_{\fS}(\wt x_1)= \pi_0 \circ \pr_{\fI_0} (h\wt y_1^\can)= \pi_0 \circ  \lb h\pr_{\fI_0}(\wt y_1^\can)\rb=\pi_0(hy_1^\can)=x_1.$$
           Thus, there exists $T_n\in\GSp_{2g}( \BA_f^{p})$        such that $x_n=\pr_{\fS}(T_n \wt x_1).$
                  Let $ y_n=  \pr_{\fI_0} (T_n\wt y_1) . $
                 Then 
                  $$y_n^\can= \pr_{\fI_0} \lb (T_n\wt y_1)^\can\rb= \pr_{\fI_0} \lb T_n(\wt y_1^\can)\rb  $$ since    $\pr_{\fI_0} $  and  $T_n$ commute with taking canonical lifting.
                Then   by \eqref{33}, we have  
                  \begin{equation*}  \label{44} \pi_0(hy_n^\can)=  \pi_0\circ h\circ \pr_{\fI_0}\circ T_n(\wt y_1^\can)=  \pr_{\fS}\circ     T_n\circ \wt \pi_0\circ h(\wt y_1^\can)=  \pr_{\fS}\circ     T_n(\wt x_1)=x_n.
\end{equation*}
Let 
$\fW'=\pi_0^{-1}(\fV) $.  Then $d(hy_n^\can,\fW')=d(x_n,\fV)\to0$.  Let 
$\fW=h^{-1}\fW'$. 
Then $d(y_n,\fW)\to0$.

Second,   we construct $Y$. Recall   $f_{j,i}: \fS_j\to \fS_i$  and  $\pi_i:\fI_0\to \fS_i$   defined in  \ref{Igusa schemes and Hecke orbits}. Claim: there exists a sequence $\{Z_i\}_{i=0^\infty}$ such that 
$Z_0$  is  an irreducible component  of $Z$, 
and $Z_{i+1}$ is  an irreducible component of   $f_{i+1,i}^{-1}(Z_i)$
containing a Zariski dense subset  of points in $\{\pi_{i}(y_n)  
\}_{n=1}^\infty$. 
 Assuming the claim, let $Y=\vpl_{i\geq 0} Z_i$. Then    condition   (1) in Proposition \ref{lem1} on $Y$ follows easily from Lemma \ref {1lem}. 
The proof of  the claim is inductive and as follows.  Assume that we have $Z_i$ as desired.
The image of  
\begin{equation}\{\pi_{i+1}(y_n)
\}_{n=1}^\infty\cap f_{i+1,i}^{-1}(Z_i).
\label{thisset}\end{equation} 
by $f_{i+1,i}$ is 
$\{\pi_{i}(y_n)  
\}_{n=1}^\infty\cap Z_i$, and Zariski dense in $Z_i$ by inductive hypothesis.
Thus the Zariski closure of  \eqref{thisset} contains  an irreducible component  
of   $f_{i+1,i}^{-1}(Z_i)$, which we choose to be $Z_{i+1}$.  

Third,    for   $y\in \pi_0^{-1}(\{x\})\cap Y$,    condition   (2) in Proposition \ref{lem1} holds.
Indeed,
by the  first part of Proposition \ref{2lem}, 
$\pi_0|_{Y_y}$ is an isomorphism to a union $Z'_x$ of formal  branches of    $Z_x$. 
Then since $Z$ is of finite type over $k$, there exists a positive integer such that $Z$ is defined over the finite field $\BF_{p^a}$
and 
$x\in Z(\BF_{p^{a}})$.   Then  $\phi_{\fS}^{a}\lb Z_x\rb \subset  Z_x$.    By the excellence and reducedness  of $Z$,  $  Z_{x} $ is reduced.  So $Y_y$ is reduced.
Moreover, by a similar (and simpler) argument as   in the proof of  Lemma \ref  {djg},
$Z'_x$ is stabilized by a power of  $\phi_{\fS}^{a}$. 
Since $ \psi_0$ is   compatible with the liftings of  Frobenii,
$Y_y$ is stabilized by the same power $\phi_{\fI}^{a}.$


Finally, by the  second part of Proposition \ref{2lem}, we can choose finitely many $y$'s as in  the third step, such that every formal  branch of    $Z_x$ is  contained in some $\pi_0|_{Y_y}$. The   lemma now follows from Proposition \ref{lem1}.    \end{proof}
     
 Now we slightly improve  Lemma \ref{thm1'}, by including canonical liftings, as follows. 
On a power of $\fS$, we can define  prime-to-$p$ Hecke action and canonical liftings in the obvious way. 
 \begin{lem}\label{thm1'1} For  non-negative integers  $a,b$, let
$\fS^{a+b}=\fS ^{a}\times \fS ^{b}$.
 Let $\fV$ be  a  formal subscheme of $\fS_{\Fo}^{a+b}$.  
 Let  $\{u_n\}_{n=1}^\infty$ be a sequence in
the prime-to-$p$ Hecke orbit of   $u_1\in \fS^a(\Fo)$ 
and $\{v_n\}_{n=1}^\infty$ be a sequence of canonical liftings in  $  \fS^b $.
Assume that  the 
reduction $\ol{x_n}$ of $x_n$ is 
in  $\fV(k)$
and $d(x_n,\fV)\to 0$. Let $Z$ be   the Zariski closure of  $\{\ol {x_n}\}_{n=1}^\infty $   in 
$\fV_{k}$.   For $x\in Z(k)$,   $\fV_x$ contains a union of  translated formal subtori that contains   $Z_x$ as a closed formal subscheme.

\end{lem} 

\begin{proof} The analog of Proposition \ref{lem1} holds for $\fI_0^{a+b}$ 
by the same proof.
Then the proof of the lemma is   essentially identical with the proof of  Lemma \ref{thm1'} (with the toric action only on the first $a$ factors). We omit the details.
\end{proof}
\subsection{Forward Frobenius action}\label{Forward Frobenius action}
The following  Theorem \ref{thm2'} is equivalent to  Theorem \ref{jointthm} (1) (the implication will be given below, and the inverse implication is easy and left to the reader).  And Theorem \ref{thm2'} has a form closer to Lemma \ref{thm1'1}, and easier to prove.

We need the following simple lemma.  
  \begin{lem}\label{as in}
  Let $R$ be a countable ring (e.g., $R$ is a   finitely generated  algebra over a countable field) and
$V=\Spec R$. For  every  non-negative integer  $i$, let $X_i\subset V $ be a subset. 

(1) Assume  that no  $X_i$ is    Zariski  dense  in $V$ and $\cup_{i=0}^\infty X_i$ is  Zariski dense in $V$.
    Then there exists an increasing  sequence $\{i_m\}_{m=0}^\infty$ of  non-negative integers and $x_{i_m}\in X_{i_m}$ such that  $\{x_{i_m}\}_{m=0}^\infty$ is Zariski dense in $V$.

(2)  Assume  that every $X_i   $ is Zariski dense   in $V$.  
Then there exists   $x_{i}\in X_{i}$ such that  $\{x_{i}\}_{i=0}^\infty$ is Zariski dense in $V$.

  \end{lem}
  \begin{proof}  Let $\{f_m\}_{m=1}^\infty$ be an enumeration  of elements in $R\bsl R^\times$.  Let $D(f_m)\subset V$ be the associated open subset. 
    For (1), by the Zariski density of the union, there exists $i_1$ and $x_{i_1}\in X_{i_1}\cap D(f_1)$. Since the union of $X_1,X_2,. . ., X_{i_1}$ is not  Zariski dense in $V$, the union of the rest of $X_i$'s is still 
  Zariski dense in $V$. Then there exists $i_2>i_1$ and  $x_{i_2}\in X_{i_2}\cap D(f_2)$. 
   Continuing this process, the lemma follows.
For (2),  choose $x_i\in X_{i}\cap D(f_i).$ 
       \end{proof}
       
  For $z\in   \fS(k)$ and a non-negative integer $m$, a quasi-canonical lifting of $z$ (of order $p^m$) is a torsion (of order $p^m$) of  the formal residue torus at $z$.
\begin{lem} [\cite{dJN}]   CM points in $\fS(\oFo)$  are the same as   quasi-canonical liftings.  
\end{lem}

   For a non-negative integer $d$ and  $I=(i^{(1)},. . .,i^{(d)})\in \BZ ^d$, let $\phi^I=\lb \phi_\fS^{i^{(1)}} ,. . . ,\phi_\fS^{i^{(d)}} \rb.$ 
\begin{thm}\label{thm2'} Let $d,d',m$ be non-negative integers.
Let $\fV$ be  a   formal subscheme of $\fS_{F^\circ}^{d+d'}$.   
Let  $\{ x_n\}_{n=1}^\infty$ be a sequence in the prime-to-$p$ Hecke orbit of   $x_1\in \fS^d(F^\circ)$,   and $\{y_n\}_{n=1}^\infty$   a sequence of   quasi-canonical liftings in  $  \fS^{d'}(\oFo)$ of    order dividing $p^m$.
For every $n$, let  $I_n=(i_n^{(1)},. . .,i_n^{(d)})\in \BZ_{\geq -m}^d$ and let $w_n\in   \phi^{I_n}\lb \{x_n\}\rb\times\{y_n\} $. 
Assume that  the
reduction $\ol{w_n}$ of $w_n$ is 
in  $\fV(k)$
and $d(w_n,\fV)\to 0$. Let $Z$ be 
the Zariski closure of  $\{\ol {w_n}\}_{n=1}^\infty $   in 
$\fV_{k}$.   For $x\in Z(k)$,   $\fV_x$ contains a union of  translated formal subtori that contains   $Z_x$ as a closed formal subscheme. \end{thm}

    \begin{proof} 
    Let $\fU$ be an open formal subscheme of $\fS_{F^\circ}^{d+d'}$ that contains $\fV$ as a closed formal subscheme. 

    \begin{lem}\label{m=0}If   Theorem \ref{thm2'} holds for $m=0$, then it holds for every   $m$.
\end{lem}
\begin{proof} Let  $\phi^m=(\phi_\fS^{m},. . . ,\phi_\fS^{m})|_{\fU}$ ) with $d+d'$-terms (see the paragraph above Corollary \ref{vpr0} (2) for  the restriction. 
By Corollary \ref{image},  $d\lb \phi^m(w_n),\phi^m(\fV)\rb\to 0$.
By   the relation between  $\phi_{\fS}$ and  the $p$  power endomorphisms of formal residue tori in
Lemma \ref{Kzz1} (2), we can apply
the case $m=0$ to obtain that
for  a general $m$, the formal completion   $  \phi^m (\fV) _{ \phi^m(x)}$  contains a translated  subtorus   that contains   $ \phi^m  (Z) _{ \phi^m(x)}$. 
Then the lemma follows from Corollary \ref{vpr0} (2).   
    \end{proof}
        
   We continue the proof of the theorem.                         By Lemma \ref{m=0}, we may   assume that $m=0$. Then $w_n=\lb  \phi^{I_n}(x_n),y_n\rb $.

          First, we   prove the theorem in the case $d'=0$.  We use Lemma \ref{as in} to simplify the problem as follows.
Up to  replacing $\{ w_n=\phi^{I_n}(x_n)\}_{n=1}^\infty$ by a subsequence, we may assume that  the set of reductions $\{ \ol {w_n}\}_{n=1}^\infty$ is contained in  an affine open   $V\subset Z$. 
          For a non-negative integer $i$,  let $X_i=\{\ol {w_n}:i_n^{(1)}=i\}. $ Then  $\cup_{i=0}^\infty X_i=\{ \ol {w_n}\}_{n=1}^\infty$, which is  Zariski dense in
        $V$ by assumption.
      By  Lemma \ref{as in}   (1), up to  replacing $\{ w_n\}_{n=1}^\infty$ by a subsequence, we may assume  one of the following two conditions:
 \begin{itemize}\item    there exists a non-negative integer  $i^{(1)}$ such that  $X_{i^{(1)}}$    is Zariski dense in $V$ and 
  $X_i=\emptyset $ for  $i\neq i^{(1)}$;
 \item     $   i_n^{(1)}\to \infty $ as $n\to \infty$.  
 \end{itemize}
Continue this process for $i_n^{(2)}$,..., $i_n^{(d)}$ and after permuting $1,...,d$, we arrive at 
a decomposition  $d=a+b$ and a sequence $\{ w_n\}_{n=1}^\infty$ with dense reduction in $Z$ such that 
\begin{itemize} 
\item for $j=1,...,a$,     $i_n^{(j)} =i_1^{(j)}$  (thus does not depend on $n$);
\item for $j=a+1,...,a+b$,  $   i_n^{(j)}\to \infty $ as $n\to \infty$.  
\end{itemize}

Now let us continue the proof of the theorem in the case $d'=0$.  
According to    $\fS^{a+b}=\fS ^{a}\times \fS ^{b}$,
    we write $w_n=(u_n,v_n)$. Since the prime-to-$p$ Hecke action and $\phi_\fS$ commute,   
$\{u_n\}_{n=1}^\infty$ is a sequence in
the prime-to-$p$ Hecke orbit of   $u_1= \phi^{I_1^{(a)}}\lb x_1^{(a)}\rb\in \fS^a(F^\circ)$, where the superscript ``$(a)$" means the first $a$-component, i.e., $I_1^{(a)}=(i_1^{(1)},. . .,i_1^{(a)}). $
Note that    $p$-power endomorphism on a formal torus over $\Fo$ shrinks the torus  toward the unit.  
Since $  i_n^{(j)}\to \infty $ as $n\to \infty$ for $j=a+1,...,a+b$, 
by the relation between  $\phi_{\fS}$ and  the $p$-power endomorphisms of formal residue tori in
Lemma \ref{Kzz1} (2),
$d(\ol {v_n}^\can, v_n)\to 0$ in  $\fS^b$.  Here 
 $\ol {v_n}$ is  the reduction 
 of $ {v_n}$. 
Let $z_n=(u_n,\ol {v_n}^\can)$.
Then $d(z_n,\fV)\to 0$. Now the theorem follows from Lemma \ref{thm1'1}.

For a general $d'$,  by the same process, we arrive at  $z_n=(u_n,\ol {v_n}^\can,y_n)$ (recall that $y_n$ is a canonical lifing) such that $d(z_n,\fV)\to 0$. The theorem follows from Lemma \ref{thm1'1}.
 \end{proof}
    \begin{lem}    \label{degree}For a non-negative integer $m $, let  $F_m$ be $F  $ adjoining all $p^m$-th roots of unity.
      For $x\in  \Fo$ such that $0<|x-1|<1$,  there exists a  positive integer $m_0$ such that for every
positive integer $m$ and
$x_m\in p^{-m}(\{x\})\bsl p^{-(m-1)}(\{x\})$, we have  $$   p^{m-m_0}\leq [F(x_m):F], [F_m:F]  \leq p^{m}.$$
 
\end{lem}
 \begin{proof}  The  second ``$\leq$"  is immediate. For the first  ``$\leq$", inspecting the Newton polygon, the ramification index of $F(x_m)$ over $F$ is $\geq p^{m-m_0}$ for some non-negative integer  $m_0$.  The same inequality for $F_m$ is well-known. 
      \end{proof}
\begin{proof}
[Proof of Theorem \ref{jointthm} (1)] 
We have a reduction step that will be used for both (1) and (2) of Theorem \ref{jointthm}.
After enlarging $F$, we may assume that each $O_i$  is the  Hecke--Frobenius orbit  of a point $o_i\in \fS(\Fo)$. 
Since the Hecke orbit of a CM point consists of CM points, 
we may assume that each $o_i$  is not a CM point (in particular, not a canonical lifting).  
For a non-negative integer $m$, define $$O_i^{(m)}:=\bigcup_{I \in \BZ_{\geq -m}^d}\phi^{I}\lb O_i\rb,$$ 
and $\CM^{(m)}\subset \CM$ the subset of quasi-canonical lifting of order dividing $p^m$.   
By  the description of    prime-to-$p$ Hecke orbits in \ref{Unramified points on Siegel moduli schemes}, every element in the  prime-to-$p$ Hecke orbit of $o_i$ have the same $\Fo$-coordinate with $o_i$ under a suitable coordinate system  (by which we mean an isomorphism to $\wh\BG_{m,\Fo}^{g(g+1)/2}$) of the corresponding formal residue torus over $\Fo$.  
By the relation between  $\phi_{\fS}$ and  the $p$  power endomorphisms of formal residue tori in
Lemma \ref{Kzz1} (2), we can apply the first ``$\leq$" of Lemma \ref{degree}  to conclude that  for $m$ large enough, depending on $e$ and $o_i$'s, \begin{equation}\label{Febound}  \lb \prod_{i=1}^d O_i\times \CM^{d'}\rb \cap \bigcup_{[E:F]\leq e}\fS (E^\circ)\subset    \prod_{i=1}^d O_i^{(m)}\times \lb\CM^{(m)}\rb^{d'} .\end{equation}

Now we can prove    Theorem \ref{jointthm} (1)  using Theorem \ref{thm2'}. 
By  assumption    there exists  a   sequence $\{ \ep_n \}_{n=0}^\infty$    of positive real numbers with  $\ep_n\to 0$, such that the Zariski  closure $Z_n$ of the reduction of  the intersection of $\fV_{\ep_n}$ with   \eqref{Febound}, which is  in $\fV_k$,    contains $Z$.
By the noetherianness of $\fV_k$, $Z_n$ stabilizes for $i$ large enough. Thus we may assume that $Z_n=Z$ for all $n$.
 By  Lemma \ref{as in}   (2) (applied with $V=Z$ and $X_i$  the reduction of  the intersection of $\fV_{\ep_{i+1}}$ with   \eqref{Febound}),    Theorem \ref{thm2'}     
implies Theorem \ref{jointthm} (1).
     \end{proof}
    \section{Backward Frobenius action} \label{Frobenius action}  
  
 
 \subsection{Boxall's trick}   
 
 We recall a result of Serban \cite{Ser21}, which has its origin in Boxall's study \cite{Box} of unlikely intersection for abelian varieties over finite fields.

              Let $\fG$ be a    formal torus 
over $F^\circ$.
For a non-negative integer $r $,   let  $F_{\cycl,r}$ be $F  $ adjoining all $p^r$-th roots of unity.
Then
$F_{\cycl,r}=F\lb \fG(\oFo)[p^r]\rb $.    
Let $$\fG(F_{\cycl,r}^\circ)^\div=\{P\in \fG(\oFo): p^nP\in \fG(F_{\cycl,r}^\circ)\text{ for some }n\in \BZ_{\geq 0}\}.$$

\begin{lem} [{\cite[Lemma 2.5]{Ser21}}]\label{Serlem}
For $r>1$  and   $P\in \fG(F_{\cycl,r}^\circ)^\div$
such that  $p^{r-1} P\not \in  \fG(F_{\cycl,r}^\circ)$,
there exists $\sigma \in \Gal(\ol F/F_{\cycl,r})$ such that $\sigma(P)-P \in  \fG(F_{\cycl,r}^\circ)[p^r] \bsl \fG(F_{\cycl,r}^\circ)[p^{r-1}]$,  
\end{lem}
\begin{rmk}    In \cite[Lemma 2.5]{Ser21},  it is required that $P$ is a torsion point. However, this requirement does not play a role in the proof of \cite[Lemma 2.5]{Ser21}. See \cite[Lemma 2.2]{Ser22}.


\end{rmk}

   \begin{cor} \label{Serprop}
Let  $r>1$, $ \Gamma \subset \fG(F_{\cycl,r}^\circ)^\div$  such that
$p^{r-1} (\Gamma)\not \in  \fG(F_{\cycl,r}^\circ)$.  
Then for every    closed formal subscheme $\fZ\subset \fG$ and $\ep>0$,   
\begin{align*} \fZ_{\ep}  \cap \Gamma& \subset   \bigcup_{Q\in \fG(F_{\cycl,r}^\circ)[p^r]\bsl \fG(F_{\cycl,r}^\circ)[p^{r-1}]}  \lb \fZ_{\ep}   \cap  \lb \fZ_{F_{\cycl,r}^\circ} +{Q } \rb_\ep \rb     \\
& \subset \lb \fZ\cap \lb \fZ+ \lb \fG [p^r]\bsl \fG [p^{r-1}]\rb\rb\rb_\ep
\end{align*}

\end{cor}
\begin{proof} For $\sigma \in \Gal(\CP/F)$, $d(P,\fZ)=d(\sigma(P),\fZ)=d(P,\fZ-\lb \sigma(P)-P\rb)$.  
Then  Lemma \ref{Serlem} implies the first ``$\subset $". 
The second ``$\subset $" follows by applying
  Lemma \ref {base}, Lemma \ref {sub} and 
Lemma \ref{internbhd}. 
\end{proof} 
The following result will be useful in order to apply Corollary \ref{Serprop}.

\begin{lem}  \label{smaller}
Let $\fZ_1,\fZ_2\subset \fG$ be   integral closed formal subschemes such that $\dim \fZ_1\leq \dim \fZ_2$. Assume that and $\fZ_1$ contains no translated formal subtori. 
Then 
for $r$ large enough, $$\dim   \fZ_1\cap  \lb\fZ_2+ \lb \fG [p^r]\bsl \fG [p^{r-1}]\rb\rb <\dim \fZ_2.$$
\end{lem}
\begin{proof}    
We prove by  contradiction. Assume  that for infinitely many  $r$, there exists $Q_r\in \fG(F_{\cycl,r}^\circ)[p^r]\bsl \fG(F_{\cycl,r}^\circ)[p^{r-1}]$ such that  $\dim \lb  \fZ_{1,F_{\cycl,r}^\circ}\cap  \lb  \fZ_{2,F_{\cycl,r}^\circ}+Q_r\rb\rb=\dim \fZ_2$. Then there exist closed formal subschemes $\fX_i$ of 
$   \fZ_{i,F_{\cycl,r}^\circ}$ of $\dim \fZ_2$ such that $\fX_1=\fX_2+Q_r$.
For $P\in   \fZ_2(\oFo)$ (which exists by Corollary \ref{samedimcor}),   we want to  find a translation of $P$  in  $ \fZ_1(\oFo)$.
 By Galois descent \cite[p 139, Example B ]{BLR} (applied to the  affine scheme corresponding to the union of Galois conjugates of $\fX_2$),  the union of Galois conjugates of $\fX_2$ is $\fZ_2$.
Then for some $\sigma\in \Gal(F_{\cycl,r}/F)$, $P\in\sigma( \fX_2)(\oFo)$. So $P+\sigma(Q_r)\in   \fZ_1(\oFo)$.  By \cite[Theorem1.3 (1)]{Ser} and the infinitude of such $r$'s, $\fZ_1$ contains a translated formal subtorus and this is a contradiction.
\end{proof}

\subsection{Proof of  Theorem \ref{jointthm}  (2)} \label{thm-5ps} 
First, we have a reduction step about our formal subscheme $\fV$ as follows.
By 
Lemma \ref{union}, we may assume that $\fV$ is flat over ${\Fo}$ after replacing $\fV$ by its maximal closed subscheme that is flat over ${\Fo}$.
       By the reduced fiber theorem \cite {BLR0}, after replacing $F$ by a finite extension and $\fV$ by a nonempty open 
       formal subscheme
       which does not affect the truth of    Theorem \ref{jointthm}  (2) by Lemma \ref{open} and 
Lemma \ref{base},   we may assume that  
       \begin{itemize}\item
       $\fV_k$ is connected, unibranch 
       and has no embedded points;
       \item there is a flat formal $\Fo$-scheme $\fV'$ with reduced unibranch special fiber $\fV'_k$ and 
    a finite schematically surjective morphism $f:\fV'\to \fV$.                 
            \end{itemize}
       These assumptions will be used when we apply  Corollary \ref           {smallercor} later.
       
   Second, we have the reduction step  about the set of points as in the  beginning of the  proof of Theorem \ref{jointthm} (1) (in the end of last section).

     Now assume that
for some   $x\in \fV(k)$, $\fV_{k,x}$   contains no translated formal subtori  of   $\fS_x^{d+d'}$, we will show that the   reduction of $\fV_{\ep}  \cap  \lb \prod_{i=1}^dO_i\times \CM^{d'}\rb$ is not Zariski dense in   of  $\fV_k$ for $\ep$ small enough. This is a contradiction and     Theorem \ref{jointthm}  (2) follows.

By Lemma \ref{smaller}, we can choose  $r$ large enough such that     
\begin{equation}           \label{smallereq}
\dim \fV_x\cap \lb \fV_x+ \lb \fS_x [p^r]\bsl \fS_x [p^{r-1}]\rb\rb<\dim \fV_x.          \end{equation}

 Let $e$ be a positive integer that is to be determined. Let $$   \Gamma  = \lb \prod_{i=1}^d O_i\times \CM^{d'}\rb\bsl \bigcup_{[E:F]\leq e}\fS (E^\circ).$$
By the second ``$\leq$"  of  Lemma \ref{degree}, we may choose (and do choose) $e$ large enough, depending on $r$, such that for every  $z\in \fS(k)$, 
$$p^{r-1}\lb \Gamma\cap \fS_z(\oFo)\rb\not\subset \fS_z(F_{\cycl,r}^\circ).$$
Then by Lemma \ref{sub2} (for  $ \fS_z \to \fS$) and   Corollary \ref{Serprop}, for $\ep>0$ and $P\in  \fV_{\ep}\cap \Gamma$, 
$$d_{\fS_z }(P,  \fV_z\cap \lb \fV_z+ \lb \fS_z [p^r]\bsl \fS_z [p^{r-1}]\rb\rb)\leq \ep ,$$
where  $z$ is the  reduction of $P$.
 By  Corollary \ref{samedimcor} (1), there exists $Q\in  \fU(\oFo)$ such that $f(Q)=P$.
Then by  Corollary \ref           {smallercor} and \eqref{smallereq},   Lemma \ref{sub}  and Lemma \ref
{sub2},
there exists a positive integer $s$       and  a closed   subscheme $ Y\subset  \fV_k$ (independent of $P,Q$) 
of $\dim Y<\dim \fV_k$
such that if 
$z\not\in Y(k)$,   then 
\begin{align*}&d_{\fS_z }(P,  \fV_z\cap \lb \fV_z+ \lb \fS_z [p^r]\bsl \fS_z [p^{r-1}]\rb\rb)
\\=
&d_{\fU_{f^{-1}(z)} }\lb Q,   f^{-1}\lb \fV_z\cap \lb \fV_z+ \lb \fS_z [p^r]\bsl \fS_z [p^{r-1}]\rb\rb\rb\rb 
\geq 
d_{\fU_{f^{-1}(z)} }(Q,   \fV'_{f^{-1}(z)}/\varpi^s )
=|\varpi^s| . \end{align*} 
Thus for    $\ep<|\varpi^s|$,
the reduction of $ \fV_{\ep}\cap \Gamma$ is contained in
$  Y$.    

  By the definition of $\Gamma$ and \eqref{Febound}, 
\begin{align}\label{either}     \fV_\ep \cap \lb \prod_{i=1}^dO_i\times \CM^{d'}\rb  \bsl \Gamma\subset   \fV_\ep \cap  \lb \prod_{i=1}^dO_i\times \CM^{d'}\rb \cap \bigcup_{[E:F]\leq e}\fS (E^\circ).\end{align} 
Since
$\fV_{k,x}$   contains no translated formal subtori  of   $\fS_x^{d+d'}$,
 by  Theorem \ref{jointthm}  (1), 
for $\ep$ small enough, 
the reduction of  the right hand side, and thus the left hand side, of  \eqref{either}   is  not Zariski dense in    of  $\fV_k$ (and is in fact finite). Thus the reduction of 
 $$\fV_{\ep}  \cap  \lb \prod_{i=1}^dO_i\times \CM^{d'}\rb=\lb \fV_{\ep}\cap \Gamma\rb \cup \lb    \fV_\ep \cap \lb \prod_{i=1}^dO_i\times \CM^{d'}\rb  \bsl \Gamma\rb$$ is not Zariski dense in   of  $\fV_k$ for $\ep$ small enough.

\appendix 
\section{Approximation on perfectoid spaces}\label{Perfectoid spaces}

We   recall some applications of Scholze's approximation lemma on  perfectoid spaces \cite{Sch12} (see  \cite{Qiu}  and  \cite{Xie}  for more applications).  

\subsection{Approximation lemma}

Let $(R,R^+)$ be a perfectoid  affinoid $(\CP,\CPo)$-algebra  with tilt the $(\Cflat,\Cfcc)$-algebra 
$(R^\flat,R^\fpl)$.
In particular, we have the $\sharp$-map \cite[Proposition 5.17]{Sch12} 
$$R^\flat\to R,\ f\mapsto f^\sharp.$$
We remind the reader that this  $\sharp$-map is not a ring homomorphism, and we do not use the explicit definition of this map. The only important fact is \eqref{impt} and Lemma \ref{Corollary 6.7. (1)}  below.

Let $\cX=\Spa(R,R^+)$ with tilt
$\cX^\flat=\Spa(R^\flat,R^\fpl)$.   Every $x\in \Spa(R,R^+)$ is an equivalence class  valuations on $R$.
Then   the tilting bijection \cite[Definition 6.16]{Sch12} \begin{equation}\label{tiltbij}\rho :\cX(\CP,\CPo)\cong \cX^\flat(\Cflat,\Cfcc),\end{equation}  that is defined by the following equation for all $g\in R^\flat$:
\begin{equation}\label{impt}|g(\rho(x))|=|g^\sharp(x)|.\end{equation}
Below, we always choose a representative $|\cdot(x)|$ in the  equivalence class  of $x$  such that  
$|u(x)|=|u|$  for every $u\in \CP$.   We make the same choice for $(\Cflat,\Cfcc)$-algebras.

\begin{lem}[{\cite[Corollary 6.7 (1)]{Sch12}}]\label{Corollary 6.7. (1)} Let $f\in R^+$.
Then   for every   $0<\ep<1$, there exists 
$g\in R^\fpl$   such that  for every $x\in \cX(\CP,\CPo)$, we have
\begin{equation}|f(x)-g^\sharp(x)|  \leq  \frac{1}{2}  \max\{|f(x)|,\ep\} = \frac{1}{2}   \max\{|g^\sharp(x)|,\ep\}.\label{2.1}\end{equation}    \end{lem}
\subsection{Algebraic setting}
Assume that $\Cflat$ is the   closure  of  an algebraic extension of ${k} ((t))$ in $\Cflat$.  
Also assume that there exists  
a  ${k}$-algebra $S$, such that  $ R^{\flat+}$ is the $t$-adic completion of $S \otimes K^{\flat\circ}$.
We have the following corollary of Lemma \ref{Corollary 6.7. (1)}.
\begin{cor}\label{improveCorollary 6.7. (1)} 
For $u\in \Cflat$ with $|u|<1$, let $m_u:=\min\{m\in \BZ :|u|^ m\leq\ep \}$.
Then  $g\in R^\fpl$,  exists 
$u\in \Cflat$ with $|u|<1$,
and    \begin{equation*}\label{range}g_\ep= \sum_{ i\in \BZ\cap [0,m_u) }g_{\ep,i}  u^i \end{equation*} with $g_{\ep,i}\in S$,
 such that    \begin{equation}g-g_\ep\in  u^{m_u}R^{\flat+}.\label{modi}
\end{equation} 

\end{cor} 
\begin {proof}
Let $n$ be a  positive integer such that $|t|^n\leq \ep.$
By the assumptions on $\Cflat$ and $S$, 
we can choose a finite extension $E={k}((u))$  of ${k}((t))$ with   $|u|<1$, and     (a  finite  sum)
$g'=\sum s_j a_j \in   S \otimes  E$, where     $s_j\in S$  and $a_j\in E$, such that 
$g-g'\in  t^n R^{\flat+}.$ 
For $a\in E$, there exist $\alpha_h\in  {k}$  such that
$$a-\sum_{ h \in \BZ \cap [0,m_u)} \alpha_h u^h \in u^{m_u} E^\circ.$$ 
Applying this to the finitely many $a_j$'s, 
the corollary follows. 
\end{proof}

      \subsection{Frobenius descent}
         \begin{asmp} \label{asmp5}  There  is
a sequence $\{S_n\to S_{n+1}\}_{n=0}^\infty$ of  morphisms of  $k$-algebras such that 
\begin{itemize}\item[(1)]$S=\vil_n S_n$;
\item[(2)] the absolute Frobenius map $S_n\to S_n,\ x\mapsto x^p,$  factors through the image of $S_{n-1}$ for $n\geq 1$. 
\end{itemize}
\end{asmp}
        Let $\cX_0 $  be the adic generic fiber of $\Spec S_0\otimes \Cfcc$. Then we have a natural morphism
${\pr^\flat}:\cX^\flat\to \cX_0$. 
Let $\Lambda\subset  \Spec S_0(k) $, and    $\Lambda^\zar_0$    the Zariski  closure of $\Lambda$ in $\Spec S_0$. 
We have the following maps and inclusions:
$$\cX(\CP,\CPo)\xrightarrow{\rho}\cX^\flat(\Cflat,\Cfcc)\xrightarrow{{\pr^\flat}}\cX_0(\Cflat,\Cfcc) \supset \Lambda^\zar  (\Cflat)\supset \Lambda    . $$

\begin{lem}\label{techprop5}  Assume Assumption \ref{asmp5}.
Let $f\in  R^+$. For $0\leq\ep< 1$, let $\Xi_\ep:=\{x\in  \cX (\CP,\CPo) :|f(x )|\leq \ep\}$.  
Assume that $ \Lambda\subset {\pr^\flat} ( \rho( \Xi_\ep))$.
Then $ ( {\pr^\flat}\circ \rho)^{-1}\left(\Lambda^\zar(\Cfcc) \right)\subset    \Xi_\ep$.
\end{lem}

    \begin{proof}  First, assume that $0<\ep<1$.
Choose   $$g_\ep= \sum_{ i\in \BZ\cap [0,m_u) }g_{\ep,i}  u^i $$
as in Corollary \ref{improveCorollary 6.7. (1)} where $g_{\ep,i}\in S$ for all $i$.
           By the assumption on $\Lambda$,  every element   $x\in  \Lambda$   can be  written as $\pr^\flat \circ \rho (y)$ where $y\in \Xi_\ep$. By  \eqref{impt}, \eqref{2.1} and  \eqref{modi},   $|g_\ep(\rho(y) )|\leq  \ep$. 
By the finiteness of the sum, we can  choose   a positive integer    $n$  such that $g_{\ep,i}\in S_{n}$ for all $i$. 
    Then $g_{\ep,i}^{p^n}\in S_0$ and   $|g_\ep^{p^n}(x )|\leq  \ep ^{p^n}$.          
Since 
$$g_\ep^{p^n}= \sum_{ i\in \BZ\cap [0,m_u) }g_{\ep,i}^{p^n}  u^{ip^n} ,$$
and  \begin{equation}\label{0or1}|g_{\ep,i}^{p^n}(x)|=0 \text{ or }1,\end{equation}  the condition $ i\in   [0,m_u)$ implies that $g_{\ep,i}^{p^n}(x)=0$. 
 Thus  $g_{\ep,i}^{p^n}$ lies in the ideal defining $\Lambda^\zar$.
  So   for every $x\in \Lambda^\zar(\Cfcc) $,  $g_{\ep,i}^{p^n}(x)=0$. Thus $g_\ep ^{p^n}(x)=0$, and then  $g_\ep (x)=0$. By  \eqref{impt}, \eqref{2.1} and   \eqref{modi},  for every  $x\in \Lambda^\zar(\Cfcc) $, we have $$|f\left(\rho^{-1}\left(\pr ^{\flat,-1}(x)\right)\right)|\leq  \ep.$$
  
  Second, the case $\ep=0$ follows by letting  $0<\ep<1$ and $\ep\to 0$.
   \end{proof}

\section{Inverse limit of schemes}         

Let $\{ X_i\}_{i=0}^\infty$ be an inverse  system of  schemes with affine transition  morphisms $f_{j,i}:X_j\to X_i$, $j\geq i$,  so that the inverse limit  $\wt X:=\vpl X_i$ exists.  Let $  \pi_i:\wt X\to X_i$ be the natural  morphism.

             \begin{lem}\label{1lem}
                   Let $ \Lambda\subset \wt X $ be a   subset and  $ \Lambda_i=  \pi_i(\Lambda)\subset X_i $.  We have the following relation between Zariski closures: \begin{equation}\Lambda ^\zar =\bigcap_{i=0}^\infty   \pi_i^{-1}\left( \Lambda_i^\zar\right) .\label{bydef}\end{equation}

\end{lem}
\begin{proof} It is enough to consider the affine case: $X_i=\Spec B_i$ and $X=\Spec B$, where $B$ is the direct limit of $B_i$'s.  
The ideal $I\subset B$   defining $\Lambda ^\zar$, with reduced induced structure as a closed subscheme, is generated by the union of the images $ I_i$ in $B$, where $I_i\subset B_i$ is the ideal of elements whose image in $B$ vanishes on $\Lambda ^\zar$.  By the definition of $\Lambda_i$, $I_i$ is the ideal defining $\Lambda_i^\zar$. Then \eqref{bydef} follows.
    \end{proof}

The rest of this appendix      is devoted to proving Proposition \ref{2lem}  about formal branches in the inverse limit with finite \etale morphisms.  We start with is a simple lemma.
             \begin{lem}\label{3lem} 
  Let $X_i=\Spec B_i$ where $B_i$ is a noetherian local ring.
          Assume that  every   $f_{j,i}$ is a  is surjective \etale  local morphism and  induces an isomorphism between residue fields.     
          Then $\pi_i$ induces to an isomorphism between the formal completions of $\wt X$ and $X_i$. 

\end{lem}
\begin{proof} By the   assumption, every $f_{j,i}^\sharp:B_i\to B_j $ is faithfully flat and thus injective. Thus we understand $B_i$'s, with maximal ideal $\fp_i$'s,  as  increasing subrings of $B$ whose union is $B$ with maximal ideal the union $\fp$ of  $\fp_i$'s.  It is enough to show that $B_i/\fp_i^n\cong B/\fp^n$ for all $n$.
By the assumption,
$f_{j,i} $ induces an isomorphism  between completions.  In particular, $B_i/\fp_i^n\to B_j/\fp_j^n$ is an isomorphism so that $B_i \to B /\fp^n$ is surjective. The kernel is  the union of $\fp_j^n\cap B_i=\fp_i^n $ for all $j\geq i$, is $\fp_i^n $.    
\end{proof}  
\begin{rmk}
The analog of the  lemma (and the the analog of  Proposition \ref{2lem} below) for henselizations instead of  formal completions also holds. However, we will use formal completions in the  main body  of the paper. For $X$  that  may not be noetherian, passing from henselization to  formal completion could be troublesome (compare with \cite[Lemma 06LJ]{Sta}.  
Thus we stick to formal completions in this   appendix.
     \end{rmk} 
     
      We need some definitions and lemmas about  (formal) branches.
\begin{defn} [{\cite[Definition 0BQ2, Definition 0C38, Lemma 00ET, Lemma 0C37]{Sta}}]\label{fbran}$\ \ \ $

(1)  A branch of a local ring is an irreducible component of the spectrum  of  its henselization.

(2)  Let $X$ be a scheme and $x\in X$. 
A branch  of   $X$ at $x$ is  a branch of  $\cO_{X,x}$.  
Let  $b(X,x)$ be the number of  branches  of $X$ at $x$. 

(3) We say that $X$ is  unibranch  at $x$ if  $b(X,x)=1$. It is  unibranch  if it is  unibranch  everywhere.

(4)  Let $\fX=\Spf A$, where $A$ is a complete noetherian local ring.  
A  formal branch of $\fX$ is a closed formal subscheme $\Spf A/\fp$  for a minimal prime ideal $\fp$ of $A$. 
\end{defn} 
\begin{lem}[{\cite[Lemma 0E20]{Sta}}]\label{0e20}
       We have    $b(X,x)=\sum_i b(X_i,x)$, where $X_i$'s are irreducible components of $X$ passing through $x$.

                 \end{lem}
    \begin{lem}[{\cite[Lemma 0C2E]{Sta}}]\label{0C2E}Assume that $X$ is noetherian and excellent,  and $x\in X$ is a closed point. Then       $b(X,x)$ is  the number of  formal branches  of the formal completion $X_x$.                 \end{lem}

\begin{lem}\label{0lem}
                 Let $X$  be   the spectrum  of a reduced excellent local ring $R$,  $x$ the closed point,  $Z\subset X $ an irreducible component so that  the formal completion $Z_x$ is a closed formal subscheme of $X_x$ by Lemma \ref{-1lem} (1). 
                 Then  $Z_x$ is a  union of formal branches of $X_x$.
                 \end{lem}

\begin{proof} We follow the proof of Lemma \ref{0e20} (i.e., \cite[Lemma 0E20]{Sta}). 
Let $\nu_Z:Z'\to Z$ and $\nu_X:X'\to X $ be the   normalization morphisms. 
Then by \cite[Lemma 02LX]{Sta},  $Z'$ is naturally a closed subscheme of $X'$, and moreover is a union of connected components of $X'$. 
Let $Z'_x$ and $X'_x$ be the corresponding completions along $\nu_Z^{-1}(x)$ and  $\nu_X^{-1}(x)$. 
Then $Z'_x$ is   a union of connected components of $X'_x$. 
By  the excellence and \cite[Lemma 0C23]{Sta}, it is not hard to show that the vertical (natural) morphisms in
  the following commutative diagram    \begin{equation*}  
\xymatrix{
	Z'_x  \ar[r]  \ar[d]   &X'_x \ar[d] \\
Z_x\ar[r]  &  X_x}.
\end{equation*}
are normalization morphisms. Thus
the left (resp. right) vertical morphism is surjective and maps   one connected component of $Z'_x$ (resp. $X'_x$)  to exactly one branch of $Z_x$ (resp. $X_x$) (see  \cite[Lemma 0C24]{Sta} and note that $Z_x,X_x$ are henselian).  The lemma follows.
    \end{proof}  

\begin{prop}\label{00lem}
               Assume that $f:Y\to X$  is   a  proper  \etale (=finite \'etale) morphism between  reduced excellent  schemes and  $x\in X$   such that for every $y\in f^{-1}(x)$,  $f$  induces an isomorphism between residue fields of $y$ and $x$. Let $Z\subset Y$ be an irreducible component. 
                 Then  $f|_{Z_y}$  is an isomorphism to           a union   of   branches of    $X_x$. And every branch
                 of $X_x$ is contained in 
                                      $f({Z_y})$  for some $y\in f^{-1}(x)$. 
       \end{prop}
\begin{proof}  
By \cite[Lemma 00ET]{Sta} (more precise, its proof), $\Spec \cO_{Z,y}$ is an irreducible component of $\Spec \cO_{Y,y}$.
By the assumptions,
$f $ induces an isomorphism $Y_y\cong X_x$  between completions.  Then by Lemma \ref{0lem}, the first part follows.
  Now we prove the second part. Let $\nu_Z:Z'\to Z$ and $\nu_X:X'\to X $ be the   normalization morphisms.
Consider the natural commutative diagram    \begin{equation}  
\xymatrix{
\coprod_{y\in f^{-1}(x)}  \coprod_{y'\in \nu_Z^{-1}(y)}	Z'_{y'}  \ar[r]  \ar[d]   &\coprod_{x'\in \nu_X^{-1}(x)}X'_{x'} \ar[d] \\
\coprod_{y\in f^{-1}(x)}Z_y\ar[r]     &  X_x}.\label{ZxXx}
\end{equation}
By the excellence and \cite[Lemma 0C23]{Sta}, the vertical morphisms are normalization morphisms. 
Then the lemma follows, as in Lemma \ref{0lem}, as long as every $x'$ is the image of some $y'$ in \eqref{ZxXx}.
But this is true since the properness of $f$   the   normalization morphisms implies that   
the natural dominant morphism  $Z'\to X'$ is proper and thus surjective.
    \end{proof}  

          Now we start to discuss   the pro-finite    analog of    Proposition   \ref{00lem}.
          Let notations be as the beginning of this appendix. 
 Assume that every $X_i $ is  noetherian, every   $f_{j,i}$ is a  \'etale.   Assume that $X_0$ is irreducible.   Let $Z_0=X_0$   and let $Z_{i+1}\subset f_{i+1,i}^{-1}(Z_i)$ be an irreducible component  inductively.
It is plain to check  by definition that 
$Z_j$ is an   irreducible component of $f_{j,i}^{-1}(Z_i)$  for any $i$ (in particular,   $Z_j$ is an  irreducible components of $X_j$). 
Let $Y =\vpl_i Z_i\subset 
     \wt X.$
Let $y\in Y$ and $x_i=\pi_i(y)\in Z_i$ so that  $f_{j,i}(x_j)=x_i$. If  $f_{j,i}$    induces an isomorphism between residue fields of $x_j$ and $x_i$,
by the \'etaleness,
  $f_{j,i} $
  induces an isomorphism between  the henselizations of $\cO_{f_{j,i}^{-1}(Z_i),x_j}$ and $\cO_{Z_i,x_i}$. 
  Then by Lemma \ref{0e20},  \begin{equation}
\label{bZZ}
b(Z_{j},x_j)\leq b(Z_{i},x_i).
\end{equation}
And if $b(Z_{j},x_j)=b(Z_{i},x_i)$, then  
$z_j$ lies in  only one irreducible component  of $f_{j,i}^{-1}(Z_i)$, which is $Z_j$, so that
$f_{j,i}|_{Z_{j,x_j}}$ is an isomorphism to  $ Z_{i,x_i}$. Moreover, if this is the case for all $i,j$, then by Lemma \ref{3lem}.
$\pi_0|_{Y_y} $ is an isomorphism to the formal completion $Z_{0,x_0}$ 

     
    \begin{prop}\label{2lem}    Assume that every $X_i $ is   further excellent and reduced, every   $f_{j,i}$ is  further proper.
    Then  for   $x_0\in Z_0$ with separably closed residue field and $y\in \pi_0^{-1}(x_0)\cap Y$, $\pi_0|_{Y_y}$  is an isomorphism to           a union   of   formal branches of    $Z_{0,x_0}$.  
Moreover, 
there exists finitely many $y\in \pi_0^{-1}(x_0)$ such that   every branch
                 of $Z_{0,x_0}$ is contained in 
                                      $\pi_0({Y_y})$  for  one of these $y$'s       
           

     

\end{prop}
\begin{proof} 



If for all $i$,  the finite subscheme $\{x_i\in f_{i,0}^{-1}(x_0):b(Z_i,x_i)=b(Z_0,x_0)\}$ of $X_i$ is nonempty,
by \eqref{bZZ}, these finite subschemes   form an inverse system.
Choose $y$ in the inverse limit of these finite subschemes.    We   are done by the discussion above the proposition.

To deal with the general case,
we do induction on   $b(Z_0,x_0)$. 
If  $b(Z_0,x_0)=1$, we   are done by \eqref{bZZ} and  the last paragraph.  For a general $b(Z_0,x_0)$, if for some $i$, $b(Z_i,x_i)<b(Z_0,x_0)$ for every $x_i\in f_{i,0}^{-1}(x_0)$, by   Proposition  \ref{00lem}, we can replace $0$ by $i$, and conclude the proposition by the induction hypothesis.   Otherwise, we   are done by the last paragraph.
   \end{proof}

\end{document}